\documentclass{ws-jaa}
\usepackage{hyperref}
\usepackage[mathscr]{euscript}
\usepackage{epsfig}
\usepackage{subfigure}
\usepackage{graphicx}
\usepackage{array}

\newtheorem{re}{Remark}[section]
\newtheorem{nota}{Notation}[section]

\begin{document}
	
	\markboth{}
	{Universal Adjacency Spectrum of (Proper) Power Graphs}
	
	%
	\catchline{}{}{}{}{}
	%
	
	\title{UNIVERSAL ADJACENCY SPECTRUM OF (PROPER) POWER GRAPHS AND THEIR COMPLEMENTS ON SOME GROUPS}
	\author{\footnotesize KOMAL KUMARI}
	
	\address{Department of Mathematics,\\ Indian Institute of Technology Kharagpur,  \\  Kharagpur, 721302, India \\
		\email{komalkumari1223w@kgpian.iitkgp.ac.in}}
	
	\author{PRATIMA PANIGRAHI}
	
	\address{Department of Mathematics,\\ Indian Institute of Technology Kharagpur,  \\  Kharagpur, 721302, India\\
		\email{pratima@maths.iitkgp.ac.in}}
	
	\maketitle
	
	
	\begin{abstract}
		The power graph $\mathscr{P}(G)$ of a group $G$ is an undirected graph with all the elements of $G$ as vertices and where any two vertices $u$ and $v$ are adjacent if and only if  $u=v^m $ or $v=u^m$, $ m \in$ $\mathbb{Z}$. For a simple graph $H$ with adjacency matrix $A(H)$ and degree diagonal matrix $D(H)$, the universal adjacency matrix is $U(H)= \alpha A(H)+\beta D(H)+ \gamma I +\eta J$, where $\alpha (\neq 0), \beta, \gamma, \eta \in \mathbb{R}$, $I$ is the identity matrix and $J$ is the all-ones matrix of suitable order. One can study many graph-associated matrices,  such as adjacency, Laplacian, signless Laplacian, Seidel etc. in a unified manner through the universal adjacency matrix of a graph. Here we study universal adjacency eigenvalues and   eigenvectors of power graphs, proper power graphs and their complements on the  group $\mathbb{Z}_n$, dihedral group ${D}_n$, and the generalized quaternion group ${Q}_n$. Spectral results of no kind for the complement of power graph on any group were obtained before. We determine the full sectrum in some particular cases. Moreover, several existing results can be obtained as  very specific cases of some results of the paper.
		
		
	\end{abstract}
	
	

	\section{Introduction}

	Here we take only simple and undirected graphs into consideration. A {\it simple undirected graph} $H$ is represented by the pair $H=(V(H), E(H))$, where $V(H)$ represents  vertex set and $E(H)$ 
	represents edge set made up of distinct, unordered pairs of  vertices. The notation $v \sim w$ for $v,w \in V(H)$ indicates that $v$ and $w$ are adjacent in $H$.  A graph is considered to be {\it $r$-regular} if each vertex has the degree  $r$.  An $n$- vertex  graph  in which every distinct pair of vertices are adjacent is known as the {\it complete graph,} and is represented by $K_n$.  
	A  subgraph $S$ of  $H$ is said to be an {\it induced subgraph} if a  pair of vertices in S are adjacent whenever the pair is adjacent  in $H$. The  {\it complement} of a graph $H$, represented by $\overline{H}$, is the graph with    $V(H)= V(\overline{H})$  and  where vertices $v$ and $w$ are adjacent if only if $v$ and $w$ are non- adjacent in $H$.
	For more graph theoretical terminologies one may refer\cite{west1996introduction}.
	
	Let us  overview  certain graph operations that  will appear in this paper. The {\it disjoint union } $G_1 \cup G_2$   of  graphs $G_1$ and $G_2$  with $V(G_1)\cap V(G_2)=\phi$, is the graph with  vertex set $V(G_1) \cup V(G_2)$ and edge set $E(G_1) \cup E(G_2)$. The {\it join} of  $G_1$ and $G_2$, represented by $G_1 \vee G_2$, is the graph  $G_1 \cup G_2$ in which  each vertex of $G_1$ adjacent to every vertex of $G_2$. For an $n-$vertex  graph $H$ with vertex set $V(H) =\{1,2,...,n\}$ and a family of vertex disjoint graphs $G_1, G_2,...,G_n$, Cardoso et al. \cite{cardoso2013spectra} defined  the $H-join$ of  $G_1, G_2,..., G_n$, denoted as $\tilde{G} = \bigvee_H \{ G_1, G_2, ...,G_n \}$, is a graph with vertex set $V(\tilde{G}) = \bigcup_{i=1}^{n}V(G_i)$  and edge set
	\begin{equation}
		E(\tilde{G})=\left(\bigcup_{i=1}^{n}E(G_i)\right) \bigcup \left(\bigcup_{ij\in E(H)}\{uv: u\in V(G_i), v\in V(G_j)\}\right)
	\end{equation} 
	Let $n_i$ be the order of $G_i$ and $v_i \in V(G_i)$. Then $deg_{\tilde{G} (v_i)} = deg_{G_i} (v_i) + \sum_{ij\in E(H)}n_j $.\\

	In this paper, $J$, $I$ and $O$ stands for the matrix of all ones, identity matrix and the zero matrix of suitable order respectively. The {\it characteristic polynomial} $det(M-\lambda I)$ for any square matrix $M$,  is represented by $\psi ({M} ; \lambda)$, and  the \textit{spectrum} of matrix $M$ is the set  of all eigenvalues (with counting multiplicities) of $M$. An eigenpair of $M$ is a pair $(\lambda,v)$ where $\lambda$ is an eigenvalue  and $v$ is its corresponding eigenvector of $M$.
	Let $H$ be an $n$-vertex graph. The {\it adjacency matrix} $A(H)$ of $H$ is an $n \times n$  matrix whose $(i,j)^{th}$ entry is equal to $1$ or $0$ according as the adjacency or non-adjacency of $i^{th}$ and $j^{th}$ vertices of $H$. 
	The {\it degree diagonal  matrix} $D(H)$ of $H$ is an $n \times n $ diagonal matrix whose  $(i,i)^{th}$ entry is the degree of $i^{th}$ vertex of $H$.
	The {\it Laplacian, signless Laplacian,}  and { \it Seidel matrices} are respectively $L(H)$ =$D(H)-A(H)$, $Q(H)$ = $D(H)$+$A(H)$ and  $S(H)$=$J-I-2A(H)$. Haemers and Omidi \cite{haemers2011universal} developed the concept of the {\it universal  matrix} of $H$, denoted  as $U(H)=\alpha A(H)+ \beta D(H)+\gamma I + \eta J$, where $\alpha(\neq 0)$, $\beta$, $\gamma$, $\eta$ $\in \mathbb{R}$.
	If $(\alpha,\beta,\gamma,\eta)$ =$(-1,1,0,0)$, $(1,0,0,0)$, $(-2,0,-1,1)$, and $(1,1,0,0)$ then $U(H)$ = $L(H)$, $A(H)$, $S(H)$, and $Q(H)$ respectively. Similarly if  $(\alpha ,\beta,\gamma,\eta)$ = $(1,-1,n,-1)$, $(-1,0,-1,1)$, $(2,0,1,-1)$, and $(-1,-1,n-2,1)$ then $U(H)$ = $L(\overline{H})$, $A(\overline{H})$, $S(\overline{H})$, and $Q(\overline{H})$ respectively. The \textit{normalized Laplacian matrix} for ${H}$ is given by $\mathcal{L}({H})=I_n-D({H})^{\frac{-1}{2}} A({H}) D({H})^{\frac{-1}{2}}$. In the case of $(\alpha,\beta, \gamma,\eta)=(-1,1-\lambda,0,0)$, we have $\psi(\mathcal{L}({H});\lambda)=\frac{\det(U({H}))}{\det(D({H}))}$. Therefore, one can explore many graph-associated matrices collectively through the universal adjacency matrix. \\

	Chakarbarty et al. \cite{chakrabarty2009undirected} developed the concept of undirected power graph of a group.  The power graph $\mathscr{P}(G)$ of a group $G$ is an undirected graph with all the elements of $G$ as vertices and two vertices $u$ and $v$ are adjacent if and only if  $u=v^m $ or $v=u^m$, $m$ $\in$ $\mathbb{Z}$. Chattopadhyay and Panigrahi \cite{chattopadhyay2015laplacian} explored Laplacian spectrum of  $\mathscr{P}(\mathbb{Z}_n)$ and $\mathscr{P}({D}_n)$. Banerjee and Adhikari \cite{banerjee2019signless} investigated  the signless Laplacian spectrum of the power graph on the group $\mathbb{Z}_n$.  The  graph  $\mathscr{P}(G)-e$  ($e$ is the identity element of $G$) represents the proper power graph  of $G$, and is denoted by $\mathscr{P}^*(G)$. Curtin et al. \cite{curtin2015punctured} studied  the structure of $\mathscr{P^*}(G)$.  Jafari and Chattopadhyay \cite{jafari2022spectrum} investigated  the spectrum of  proper power graph of the direct product of certain finite groups. Cardoso et al. \cite{cardoso2013spectra} worked on the adjacency spectrum of  H-join of regular graphs. Wu et al. \cite{wu2014signless} discussed signless Laplacian spectrum of  H-join of regular graphs. Haemers and Obudi \cite{haemers2020universal} determined  the universal adjacency  characteristic polynomial  of the disjoint union of regular graphs. Saravanan et al. \cite{saravanan2021generalization} determined  the universal adjacency characteristic polynomial of H-join of graphs in terms  of the component graphs and determinant of a non-symmetric matrix.  Bajaj and Panigrahi \cite{bajaj2022universal}  obtained the universal adjacency spectrum of H-join of graphs in terms of the adjacency spectrum of the component graphs and spectrum of a symmetric matrix. For more results on power graphs one may refer the recent survey  \cite{kumar2021recent}.  Here we explore the universal adjacency spectrum with corresponding eigenvectors of the power graph on  $\mathbb{Z}_n$, the dihedral group ${D}_n$ and the generalized quarternion group ${Q}_n$. It is to be noted that any kind of eigenvectors of power graph on any group were not discussed before. Also, no spectral results of the  complement of power graphs were existed in the literature. Here we obtain the universal adjacency spectrum and corresponding eigenvectors of the complement of power graphs on the above mentioned groups. Finally, we discuss universal adjacency spectral results of the proper power graph on $\mathbb{Z}_n$, ${D}_n$ and ${Q}_n$.
	
	
	\section{ Power Graph of the Group  $\mathbb{Z}_n$}\label{2}
	
	We consider the group $\mathbb{Z}_n=\{0,1,\cdots,(n-1)\}$, where addition modulo $n$ is the binary operation. So, $\mathscr{P}(\mathbb{Z}_n)$ has  $n$  vertices. Number of generators in  group  $\mathbb{Z}_n$  is $\phi(n)$, where $\phi$ is the Euler's totient function. 
	We assume  the  prime factorization  of $n$ as ${p_1}^{\beta_1} {p_2}^{\beta_2}\ldots {p_\xi}^{\beta_\xi}$. 
	The  number of distinct positive divisors of $n$ is given by:
	
	\begin{equation}\label{properdivisor}
		t = \prod\limits_{r=1}^{\xi} (\beta_r+1)
	\end{equation}

	We assume $d_1,d_2,\cdots,d_{t}$ as the distinct divisors of $n$.  Let $H_{d_i}=\{x\in \mathbb{Z}_n : gcd(x,n)=d_i\}$,  $1\leq i \leq t$. In other words, $H_{d_i}=\{0\}$ if $d_i=n$, and $H_{d_i}=\{md_i : 0 < m < \frac{n}{d_i}, \gcd(m, \frac{n}{d_i}) = 1\}$ if $d_i\neq n$. Then   $H_{d_1},H_{d_2},\ldots,H_{d_t}$  be a partition of $V(\mathscr{P}(\mathbb{Z}_n))$, that is 
	\begin{equation}\label{vertexset}
		V(\mathscr{P}(\mathbb{Z}_n))=H_{d_1}\cup H_{d_2}\cup\ldots \cup H_{d_{t}}
	\end{equation}
	
	\begin{lemma}\label{lemma1} (\cite{young2015adjacency}, Proposition 2.1)
		For   divisor $d_i$ of $n$, $|H_{d_{i}}|=\phi(\frac{n}{d_{i}})$.
	\end{lemma}
	
	We note that if $d_i=n$  then the vertex in $H_{d_i}=\{0\}$ is adjacent to all other vertices in $\mathscr{P}(\mathbb{Z}_n)$.
	
	\begin{proposition}\label{lemma2}
		For $i,j \in \{1,2,\ldots,t\}$, and $d_i, d_j \neq n$, a vertex of $H_{d_i}$ is adjacent to a vertex in $H_{d_j}$  if and only if either $d_i$ divides $d_{j}$ or $d_j$ divides $d_{i} $. 
	\end{proposition}
	
	\begin{proof}
		First let $ x\in H_{d_i}$ and $y\in H_{d_j}$ be adjacent in $\mathscr{P}(\mathbb{Z}_n)$.
		As we are considering additive group $ \mathbb{Z}_n $, $x$ is either integer multiple of $y$ or  $y$ is integer multiple of $x$. Without loss of generality, let $x=ly$, where $l$ is an integer.
		We have $x = m_{1}d_{i}$ and  $y = m_{2} d_{j}$ for some integers $ m_{1}$, $m_{2}$, with $0 < {m_1} < \frac{n}{d_i}$ , $0 < {m_2} < \frac{n}{d_j}$. And
		\begin{equation}\label{equation4}
			\gcd \left(m_{1}, \frac{n}{d_i}\right) = 1 = \gcd \left(m_{2}, \frac{n}{d_{j}}\right)
		\end{equation}
		Now $m_1d_i= lm_2 d_j$, and so $\frac{m_1d_i}{d_j}$ is an integer. Then  
		$ \frac{m_1d_i}{d_j} = \frac{m_1d_i.n}{d_j.n}= \frac{m_1.(\frac{n}{d_j})}{(\frac{n}{d_i})}$ is an integer. By (\ref{equation4}) $ \gcd \left(m_{1}, \frac{n}{d_i}\right) = 1 $. So   
		$\frac{n}{d_i}$ divides $\frac{n}{d_j}$ and then  we get  $d_j$ divides $d_i$.\\
		
		Conversely, without loss of generality, let  $d_j|d_i$. Let $y = m_{2} d_{j}$ be an arbitrary vertex of $H_{d_j}$. We show that $y$ is adjacent to every  vertex in $H_{d_i}$. An vertex $x$ of $H_{d_i}$ can be written as $x=md_i$, where $ 0 < m < \frac{n}{d_i}$ and $ \gcd \left(m, \frac{n}{d_i}\right) = 1$. 
		
		We claim  $gcd\left(m_2, \frac{n}{d_i} \right)=1$. If not, then let  $gcd \left(m_2, \frac{n}{d_i} \right)= d \neq 1$. Since  $d_j$ divides $d_i$,   $\frac{n}{d_i}$ divides $\frac{n}{d_j}$.  So $\frac{n}{d_j}= r\frac{n}{d_i}$, for some integer $r$. Now $1 = gcd\left(m_2, \frac{n}{d_j}\right) = gcd\left(m_2, r\frac{n}{d_i}\right)$ which contradicts the assupmtion that  $gcd\left(m_2, \frac{n}{d_i}\right)= d \neq 1 $. Hence the claim is true, and  
		\begin{equation}\label{eq5}
			gcd\left(m_2, \frac{n}{d_i}\right)= 1 
		\end{equation}
		
		Then we get \begin{equation}\label{eq6}
			gcd\left(m m_2, \frac{n}{d_i}\right)= 1
		\end{equation}
		
		Let us take $l= \frac{m d_i }{d_j}$. Since $d_j|d_i$, $l$ is an integer. Then  
		$ly= \frac{m_2 d_j m d_i}{d_j}= m_2 m d_i $. Let $r_m$ be the remainder when $m_2 m$ divided by $\frac{n}{d_i}$. So  $r_m < \frac{n}{d_i}$  and by (\ref{eq6}), $r_m \neq 0$. Also $m_2 m=r_m~ \left ( mod ~(\frac{n}{d_i})\right)$.  Now       
		$ly= m_2 m d_i =r_m d_i (mod~n)$, and $r_m d_i<n$. Since $m_2 m = q(\frac{n}{d_i})+ r_m$, $q$ is an integer, by (\ref{eq6}) we get  $\gcd(r_m, \frac{n}{d_i}) = 1$. So $r_m d_i \in H_{d_i}$ and $y$ is adjacent with $r_m d_i$ in $\mathscr{P}(\mathbb{Z}_n)$.
		
		Next we show that $\{r_m: m\in S\}=S$, where $S=\{ m: 0 < m < \frac{n}{d_i}, \gcd(m, \frac{n}{d_i}) = 1\}$. In other words, whenever $m$ runs over $S$, $r_m$ also takes all the values in $S$, that is different $m$ corresponds to different $r_m$. Suppose for $m, m'\in S$ with $m \neq m'$, $r_m=r_m'$. Without loss of generality, let $m < m'$. We have
		\begin{equation}\label{equationa}
			m_2 m =  \left(\frac{n}{d_i}\right) q + r_{m}  
		\end{equation}
		\begin{equation}\label{equationl}
			m_2 m' =  \left(\frac{n}{d_i}\right) q' + r_{m'}
		\end{equation}
		where $q'$ is an integer.
		Subtracting (\ref{equationa}) from (\ref{equationl}), we get     
		$ m_2(m' - m ) = \left(\frac{n}{d_i}\right)( q' -  q)$. So
		$\left(\frac{n}{d_i}\right) $ divides $m_2(m' - m )$. Then by (\ref{eq5}), $\left(\frac{n}{d_i}\right)$ divides  $(m' - m )$, and this is a contradiction because $ 0 < (m' - m ) < \frac{n}{d_i}$. So $\{r_m: m\in S\}=S$ and $\{r_m d_i: m\in S\}= H_{d_i}$. Hence $y$ is adjacent with all the vertex of $H_{d_i}$.
	\end{proof}
	\begin{corollary}\label{corollary1}
		$H_{d_{i}}$ induces  $K_{\phi\left(\frac{n}{d_{i}}\right)}$ in
		$\mathscr{P}(\mathbb{Z}_n)$, for $1 \leq i \leq t$.
		
	\end{corollary}
	\begin{proof} Immediate from Lemma \ref{lemma1} and Proposition \ref{lemma2}.
	\end{proof}
	Now we consider a simple graph $\Omega_n$, where the vertex set equal to  all positive divisors of $n$, that is $V(\Omega_n)= \{ d_1,d_2,\cdots,d_{t} \}$, and where two distinct vertices are adjacent if and only if one divides the other.  From (\ref{vertexset}) and Proposition \ref{lemma2}, we get that 
	\begin{equation}\label{eq7}
		\mathscr{P}(\mathbb{Z}_n) \cong \bigvee_{\Omega_n} \{\mathscr{P}(H_{d_1}), \mathscr{P}(H_{d_2}),\ldots, \mathscr{P}(H_{d_{t}}) \}
	\end{equation}
	By Corollary \ref{corollary1}, degree of a  vertex $v \in H_{d_{i}}$ in $\mathscr{P}(\mathbb{Z}_n)$ is  given below:
	\begin{equation}\label{upsilon}
		deg~v =  		
		\phi \left(\frac{n}{d_i} \right)-1 +	\sum\limits_{\{d_{i},d_{j}\} \in E(\Omega_n)} \phi \left(\frac{n}{d_j} \right)  	
	\end{equation}
	The following result will be referred  in the sequel. 		 
	\begin{theorem}\label{theorem2.1} (\cite{chakrabarty2009undirected},Theorem 2.12)
		Let G be a finite group. Then $\mathscr{P}(G)$ is complete if and only if $G$ is the cyclic group of order $1$ or $p^r$, where $p$ is any prime and $r \in \mathbb{N}$.
	\end{theorem}
	\section{ Universal Adjacency Spectrum of $\mathscr{P}(\mathbb{Z}_n)$}\label{3}
	
	The following theorem provides the universal adjacency spectrum of an  $H$-join graph $\tilde{G} = \bigvee_H \{H_1, H_2, \ldots, H_k \}$ in terms of the spectra of graphs $H_i$, $1\leq i \leq k$, and a symmetric matrix of order $k \times k$.
	
	\begin{theorem}\label{theorem3.1} (\cite{bajaj2022universal}, Theorem 3.1.1)\label{main}
		Let $H$ be a graph with $V(H)=\{1, 2, \ldots, k\}$ and $H_1, H_2, \ldots,H_k$ be pairwise disjoint graphs on  $n_1, n_2, \ldots, n_k$ number of vertices respectively. Let the graph $H_i$ be $r_i$-regular for $1\leq i \leq k$. Let $\rho_{j}$ = $\sum\limits_{ij \in E(H)} n_{j}$. Let the adjacency eigenvalues of $H_i$ be $\lambda_{i,1}\geq \lambda_{i,2}\geq \ldots \lambda_{i,n_{i}}$ with the associated eigenvectors $\textbf{u}_{i,1},\textbf{u}_{i,2},\ldots,\textbf{u}_{i,n_{i}}$, for $1\leq i \leq k$, respectively. Then the universal adjacency eigenvalues and eigenvectors of $\bigvee_H \{H_{1},H_{2},...,H_{k}\}$ consist of:\\
		$\textbf{1.}$ The eigenvalues  $\Lambda_{i,j}$  = $\alpha \lambda_{i,j} + \beta (r_{i} +\rho_{i}) + \gamma$ with the associated eigenvector $\textbf{X}_{i,j}$ = $(\textbf{x}_1 , \textbf{x}_2, \ldots,\textbf{x}_k) ^{T}$,  where 
		\begin{equation*}
			\textbf{x}_l =\begin{cases}
				\textbf{u}_{i,j} ^{T}  ,             & \text{if~~}  l=i,\\
				0    ,             & \text{else},
			\end{cases} \text{~~where i= 1,2,\ldots,
				k~~} \text{ and j= 2,3,\ldots,$n_i$~~}
		\end{equation*}
		$\textbf{2.}$ The eigenvalue $\lambda_{i}$ of the symmetric matrix $\mathbb{K}$  with the corresponding eigenvector 
		$(\nu_{i,1} \sqrt\frac{n_{k}}{n_{1}} \textbf{j} _{n_{1}}, \nu_{i,2} \sqrt\frac{n_{k}}{n_{2}} \textbf{j} _{n_{2}},\ldots,\nu_{i,k-1} \sqrt\frac{n_{k}}{n_{k-1}} \textbf{j} _{n_{k-1}}, \nu_{i,k} \textbf{j}_{n_{k}}  ) ^{T} $ 
		where $\nu_{i}$ = $( \nu_{i,1}, \nu_{i,2},\ldots,\nu_{i,k} )^{T} $ and $(\lambda_{i}, \nu_{i})$ is an eigenpair of $\mathbb{K}$ which is given below:
		\begin{equation}\label{symmetricmatrix}
			\mathbb{K} =
			\begin{pmatrix}
				\kappa_{1}&\theta_{1,2} \sqrt{n_{1} n_{2}}&\cdots &\theta_{1,k} \sqrt{n_{1} n_{k}} \\
				\theta_{2,1} \sqrt{n_{2} n_{1}}&\kappa_{2}&\cdots &\theta_{2,k} \sqrt{n_{2} n_{k}} \\
				\vdots & \vdots & \ddots & \vdots\\
				\theta_{k,1} \sqrt{n_{k} n_{1}}&\theta_{k,2} \sqrt{n_{k} n_{2}}&\cdots &\kappa_{k}
			\end{pmatrix}
		\end{equation}
		where $\kappa_{i}$ = $\alpha r_{i}+\beta(r_{i}+ \rho_{i})+\gamma +\eta n_{i}$ and 
		\begin{equation*}
			\theta_{i,j} =\begin{cases}
				\alpha +	\eta   ,             & \text{if~~}  i j \in E(H) , \\ 
				\eta       ,          & \text{else},    
			\end{cases}  1\leq i,j \leq k.
		\end{equation*}
	\end{theorem}

	\begin{nota}
		{\rm For integers $n$ and $r$, $e_{n,r}^T$ denotes the column vector of dimension $n$, where the  first entry is equal to $1$, $r$th entry  is equal to $-1$ and the remaining entries are equal to $0$.  The all one row vector of dimension $n$ is denoted by $\textbf{j}_n$.  By notation $a |b$ we mean  that $a$ divides $b$.}
	\end{nota}
	
	From  Corollary \ref{corollary1}, Lemma \ref{lemma1}, (\ref{eq7}), (\ref{upsilon}), and Theorem \ref{theorem3.1}, we get the result below which gives eigenpairs of $ U(\mathscr{P}(\mathbb{Z}_n))$.
	
	\begin{theorem} \label{theorem3.2} The  eigenvalues of $ U(\mathscr{P}(\mathbb{Z}_n))$  and corresponding eigenvectors are as given below:
		\begin{enumerate}
			\item  The eigenvalue $\Lambda_{i}$ = $-\alpha + \Big( \phi(\frac{n}{d_i})- 1+\sum\limits_{d_{i} \sim d_{j}}\phi\left(\frac{n}{d_{j}}\right) \Big)\beta +\gamma$ with multiplicity $\left( \phi(\frac{n}{d_{i}}) - 1\right)$. For fixed $i$ $(1 \leq i \leq t)$, the linearly independent eigenvectors associated with $\Lambda_i$ are $X_{i,r}$ = $(x_1,x_2,\ldots,x_t)^T,$ where  
			\begin{equation*}
				\textbf{x}_k =\begin{cases}
					\textbf{e}_{\phi(\frac{n}{d_i}),r} ^{T},              & \text{if~~}  k=i,\\
					0             ,    & \text{else},
				\end{cases} \text{~~where i= 1,2,\ldots,
					t~~} \text{ and r= 2,3,\ldots,$\phi \left(\frac{n}{d_i}\right)$~~}
			\end{equation*}
			
			\item  The rest of  $t$ eigenvalues of $U(\mathscr{P}(\mathbb{Z}_n))$ are the eigenvalues of  $\mathbb{K}$  with associated eigenvectors   $(\nu_{i,1} \sqrt\frac{\phi(\frac{n}{d_t})}{\phi(\frac{n}{d_1})} \textbf{j} _{\phi(\frac{n}{d_1})},$ $ \nu_{i,2} \sqrt\frac{\phi(\frac{n}{d_t})}{\phi(\frac{n}{d_2})} \textbf{j} _{\phi(\frac{n}{d_2})}$ $,\ldots,$ $ \nu_{i,t-1} \sqrt\frac{\phi(\frac{n}{d_t})}{\phi(\frac{n}{d_{t-1}})} \textbf{j} _{\phi(\frac{n}{d_{t-1}})},$ $ \nu_{i,t} \textbf{j}_{\phi(\frac{n}{d_t})} )^{T},$ 
			where $\nu_{i}$ = $( \nu_{i,1}, \nu_{i,2},\ldots,\nu_{i,t} )^{T} $ and $(\lambda_{i} ,\nu_{i})$ is an eigenpair of $\mathbb{K}$ which is given below:
		\end{enumerate} 
		
		\begin{equation}\label{symmetricmatrix1}
			\mathbb{K} =
			\begin{pmatrix}
				\kappa_{1}&\theta_{1,2} \sqrt{\phi(\frac{n}{d_1}) \phi(\frac{n}{d_2})}&\cdots &\theta_{1,t} \sqrt{\phi(\frac{n}{d_1}) \phi(\frac{n}{d_t})} \\
				\theta_{2,1} \sqrt{\phi(\frac{n}{d_2}) \phi(\frac{n}{d_1})}&\kappa_{2}&\cdots &\theta_{2,t} \sqrt{\phi(\frac{n}{d_2}) \phi(\frac{n}{d_t})} \\
				\vdots & \vdots & \ddots & \vdots\\
				\theta_{k,1} \sqrt{\phi(\frac{n}{d_t}) \phi(\frac{n}{d_1})}&\theta_{t,2} \sqrt{\phi(\frac{n}{d_t}) \phi(\frac{n}{d_2})}&\cdots &\kappa_{t}
			\end{pmatrix}
		\end{equation}
		where $\kappa_{i}$ = $\alpha \left( \phi(\frac{n}{d_{i}} ) -1 \right) +\beta \Big(\phi(\frac{n}{d_i})-1 + \sum\limits_{d_{i} \sim d_{j}} \phi(\frac{n}{d_{j}}) \Big)+\gamma + \phi\left(\frac{n}{d_{i}} \right)\eta $  and 
		\begin{equation*}
			\theta_{i,j} =\begin{cases}
				\alpha +	\eta,              & \text{if~~}  d_i |d_j \text{~or~} d_j |d_i , \\ 
				\eta          ,      & \text{else}, 
			\end{cases}   1\leq i,j \leq t.
		\end{equation*}
	\end{theorem}
	
	\begin{proof}
		By Corollary \ref{corollary1}, $\mathscr{P}(H_{d_i})$ $\cong $ $K_{\phi\left(\frac{n}{d_{i}}\right)}$,  and so $\mathscr{P}(H_{d_i})$ is a regular graph of degree $\phi(\frac{n}{d_i})-1$ for each $i$, $1\leq i \leq t$. From equation (\ref{eq7}),
		$	\mathscr{P}(\mathbb{Z}_n) \cong \bigvee_{\Omega_n} \{\mathscr{P}(H_{d_1}), \mathscr{P}(H_{d_2}),\ldots, \mathscr{P}(H_{d_{t}})\}$. Now the required  can be obtained  by (\ref{upsilon}) and Theorem \ref{theorem3.1}.
	\end{proof}
	
	The corollary below follows directly from Theorem \ref{theorem2.1} and Theorem \ref{theorem3.2}. 
	
	\begin{corollary}\label{corollary3.1}
		If  $ n = p^{r} $, the  spectrum  of $U(\mathscr{P}(\mathbb{Z}_n))$ consists of the eigenvalues
		$(\alpha (p^{r} -1) + \beta (p^{r} -1) + \eta p^{r} +\gamma )$ and $( -\alpha + \beta (p^{r}-1) +\gamma)$ with multiplicity $1$ and $(p^{r}-1)$ respectively. Moreover,  the corresponding eigenvectors are $\textbf{j}_{n} ^{T}$ and $e_{n,l} ^{T}$, where $l= 2, 3,\ldots,n$. 
	\end{corollary}	
	\begin{re}
		{\rm By assuming $(\alpha, \beta,\gamma, \eta) = (-1,1,0,0)$ and $(1,1,0,0)$ in  Corollary \ref{corollary3.1} one gets the results of (\cite{chattopadhyay2015laplacian}, Corollary 3.3) and 	(\cite{banerjee2019signless}, Theorem 3.1) respectively.}
	\end{re}
	\begin{re}
		{\rm Assuming $(\alpha, \beta, \gamma, \eta)= (1,1,0,0)$ and $(-1,1,0,0)$ in  Theorem \ref{theorem3.2}, one gets the results of 	(\cite{banerjee2019signless}, Theorem 4.2) and (\cite{chattopadhyay2015laplacian}, Theorem 2.5) respectively.} 	 
	\end{re}
	
	
	In the next theorem we explore some more eigenvalues of $U(\mathscr{P}(\mathbb{Z}_{n}))$, if $n$ is product of two distinct primes. Morever, we are able to establish the characteristic polynomial of corresponding $\mathbb{K}$ in terms of determinant of a symmetric tridiagonal matrix (except the case $\eta +\alpha \neq 0 $ with $\eta \neq 0 $ ).
	
	\begin{theorem}\label{theorem3.3}
		For  distinct primes $p$ and $q$, the  eigenvalues of $ U(\mathscr{P}(\mathbb{Z}_{pq}))$  and corresponding eigenvectors are as given below: 
		\begin{enumerate}
			\item The eigenvalue $\Lambda_{i}=$ $ -\alpha + \big( \phi(\frac{pq}{d_i})-1+ \sum\limits_{d_{i} \sim d_{j}} \phi(\frac{pq}{d_{j}})\big)\beta +\gamma$  with multiplicity $m_{i}$ = $\phi(\frac{pq}{d_{i}}) -1$  with the associated eigenvectors $\textbf{X}_{i,r}$ = $(\textbf{x}_1, \textbf{x}_2 , \textbf{x}_2 ,\textbf{x}_4) ^{T}$, where 
			\begin{equation*}
				\textbf{x}_l =\begin{cases}
					\textbf{e}_{\phi(\frac{pq}{d_{i}}),r} ^{T},               & \text{if~~}  l=i,\\
					0        ,         & \text{otherwise},
				\end{cases} \text{~~where i= 1,2,3,
					4~~} \text{ and r= 2,3,\ldots,$\phi \left(\frac{pq}{d_{i}}\right)$~~}
			\end{equation*}
			
			\item  And the final four eigenvalues of $U(\mathscr{P}(\mathbb{Z}_{pq}))$ are the eigenvalues of  $\mathbb{K}$  with associated eigenvectors   $(\nu_{i,1} \sqrt\frac{\phi(\frac{pq}{p})}{\phi(\frac{pq}{pq})} \textbf{j}_{1}, $ $ \nu_{i,2} \sqrt\frac{\phi(\frac{pq}{p})}{\phi(\frac{pq}{1})} \textbf{j} _{\phi(pq)}, $ $  \nu_{i,3} \sqrt\frac{\phi(\frac{pq}{p})}{\phi(\frac{pq}{q})} \textbf{j} _{\phi(q)}, $ $ \nu_{i,4} \textbf{j}_{\phi(p)}  ) ^{T} $ 
			where $\nu_{i}$ = $( \nu_{i,1}, v_{i,2},\nu_{i,3}, \nu_{i,4} )^{T} $, $(\lambda_{i}, \nu_{i})$ is an eigenpair of $\mathbb{K}=(\kappa_{ij})$ as described below in four cases:
			
		\end{enumerate}
		
		$\mathbf{Case~ 1.}$ 	If $ \alpha +	\eta  = 0$ and  $\eta \neq 0 $  then \\
		\begin{equation*}
			\kappa_{ii} =\begin{cases}
				\beta \left( \phi(\frac{pq}{d_i})-1 + \sum\limits_{d_{i} \sim d_{j}} \phi\left(\frac{pq}{d_{j}}\right)  \right) +\gamma +\eta                           
			\end{cases}
		\end{equation*}
		\begin{equation*}
			\kappa_{ij} =\begin{cases}
				0,   & \text{if~~}  d_i |dj \text{~or~} d_j |d_i \\ 
				\eta\sqrt{\phi(p){\phi(q
						)}},            & \text{otherwise~~~~~~~}  
			\end{cases} \text{ i $\neq$ j}
		\end{equation*}
		Also the eigenvalues are $\lambda_{1}$ = $	\beta( pq -1 ) +\gamma +\eta$, $\lambda_{2}$ =  $	\beta( pq -1 ) +\gamma +\eta$,
		$ \lambda_{3} =\frac{ \beta(2pq-p-q) +2(\eta +\gamma) + \sqrt{\beta^2 (p-q)^2 + 4\eta^2(p-1)(q-1)}}{2}$  and $\lambda_{4} = \frac{ \beta(2pq-p-q) +2(\eta +\gamma) - \sqrt{\beta^2 (p-q)^2 + 4\eta^2(p-1)(q-1)}}{2}$\\
		$\mathbf{Case 2.}$ 	If $\eta = 0 $ and $ \alpha +\eta  \neq 0$ then \\ 
		\begin{equation*}
			\kappa_{ii} =\begin{cases}
				\alpha(\phi(\frac{pq}{d_{i}}) -1) + \beta \left(\phi(\frac{pq}{d_i})-1+ \sum\limits_{d_{i} \sim d_{j}} \phi(\frac{pq}{d_{j}})  \right) +\gamma              
			\end{cases}
		\end{equation*} and
		\begin{equation*}
			\kappa_{ij} =\begin{cases}
				\alpha\sqrt{\phi(\frac{pq}{d_{i}}) \phi(\frac{pq}{d_{j}})}            & \text{if~~}  d_i |d_j \text{~or~} d_j |d_i  \\ 
				0                & \text{otherwise~~~~~~~}  
			\end{cases}  \text{ i $\neq$ j}
		\end{equation*} \\ The characteristic polynomial is 
		\begin{equation*}
			\begin{split}
				\psi (\mathbb{K} ; \lambda)  &= \phi(1)\phi(pq)\phi(p)\phi(q) \Biggl\{ 
				\frac{(\kappa_{11}-\lambda)(\kappa_{22} - \lambda)(\kappa_{33}-\lambda)(\kappa_{44}- \lambda)}{\phi(1)\phi(pq)\phi(p)\phi(q)} 
				+ \frac{2\alpha (\kappa_{22} - \lambda)(\kappa_{33}-\lambda)(\kappa_{44}- \lambda)}{\phi(pq)\phi(p)\phi(q)}\\
				&-\alpha^2\frac{(\kappa_{11} - \lambda)(\kappa_{44} - \lambda)}{\phi(1)\phi(q)}
				-\alpha^2\frac{(\kappa_{11} - \lambda)(\kappa_{33} - \lambda)}{\phi(1)\phi(p)}
				- \alpha^2\frac{(\kappa_{22} - \lambda)(\kappa_{44} - \lambda)}{\phi(pq)\phi(q)} \\
				&- \alpha^2\frac{(\kappa_{33} - \lambda)(\kappa_{22} - \lambda)}{\phi(pq)\phi(p)} - \alpha^2\frac{(\kappa_{11} - \lambda)(\kappa_{44} - \lambda)}{\phi(p)\phi(q)} \Biggl\} 
			\end{split}
		\end{equation*}
		\\
		$\mathbf{Case 3.}$ 	If $ \alpha +	\eta  = 0$ and $\eta = 0 $, $U(\mathscr{P}(\mathbb{Z}_{pq}))$ is undefined.  \\ 
		$\mathbf{Case 4.}$ 	If  $ \alpha +	\eta  \neq 0$  and $\eta \neq 0 $  then
		\begin{equation*}
			\kappa_{ij} =\begin{cases}
				( \alpha +	\eta )\sqrt{\phi(\frac{pq}{d_i}){\phi(\frac{pq}{d_j})}}             & \text{if~~}  d_i |d_j \text{~or~}  d_j |d_i   \\ 
				\eta\sqrt{\phi(\frac{pq}{d_i}){\phi(\frac{pq}{d_j})}}                  & \text{otherwise~~~~~~~}  
			\end{cases}  \text{ i $\neq$ j}
		\end{equation*}
		
		\begin{equation*}
			\kappa_{ii} =\begin{cases}
				\alpha(\phi\left(\frac{pq}{d_{i}}\right) - 1) + \beta( \phi(\frac{pq}{d_i}) + \sum\limits_{d_{i} \sim d_{j}} \phi\left(\frac{pq}{d_{j}}\right) - 1) +\gamma +\eta \phi\left(\frac{pq}{d_{i}}\right)              
			\end{cases}
		\end{equation*}
	\end{theorem} 
	\begin{proof}
		$pq$ has the distinct divisors $1$, $p$, $q$ and $pq$. Then $(1)$ follows directly from Theorem \ref{theorem3.2}. The columns and rows  of $\mathbb{K}$ are indexed in the following order  $d_{1}= pq $, $d_{2}= 1$, $ d_{3}= q$, $ d_{4}= p$. Next we consider the four cases of $(2)$.\\
		$\mathbf{Case 1.}$ 	If $\eta \neq 0 $ and $ \alpha +\eta  = 0$,  we get,
		\begin{equation*}
			\mathbb{K} =
			\begin{pmatrix}
				\kappa_{11}&0&0 &0 \\
				0&\kappa_{22}&0  &0 \\
				0 & 0 & \kappa_{33} & \eta\sqrt{\phi(q)\phi(p)}\\
				0&0&\eta\sqrt{\phi(p)\phi(q)} &\kappa_{44}
			\end{pmatrix} 
		\end{equation*}
		where  $\kappa_{ii}$ = $\alpha \left( \phi(\frac{pq}{d_{i}} ) -1 \right) +\beta \left(\phi(\frac{pq}{d_i}) -1+ \sum\limits_{d_{i} \sim d_{j}} \phi(\frac{pq}{d_{j}} ) \right)+\gamma +\eta \phi\left(\frac{pq}{d_{i}} \right)$. \\
		
		Now the characteristic  polynomial   is 
		\begin{equation*}
			\psi (\mathbb{K} ; \lambda) =	\vert\mathbb{K} - \lambda I\vert =
			\begin{vmatrix}
				\kappa_{11}- \lambda&0&0 &0 \\
				0&\kappa_{22}- \lambda&0  &0 \\
				0 & 0 & \kappa_{33}- \lambda & \eta\sqrt{\phi(q)\phi(p)}\\
				0&0&\eta\sqrt{\phi(p)\phi(q)} &\kappa_{44}- \lambda 
			\end{vmatrix}	  			
		\end{equation*}
		$~~~~~~~~~~~~~~~~~~~~~~~= (\kappa_{11}- \lambda)(\kappa_{22}- \lambda)[(\kappa_{33}- \lambda)(\kappa_{44}- \lambda)-\eta^2(p-1)(q-1)]. $\\
		Then the spectrum of $\mathbb{K}$ comprises of:\\
		$\lambda_{1} = \kappa_{11}$; $\lambda_{2} =\kappa_{22}$;
		$ \lambda_{3} =\frac{ \beta(2pq-p-q) +2(\eta +\gamma) + \sqrt{ 4\eta^2(p-1)(q-1) +\beta^2 (p-q)^2 }}{2}$  and $\lambda_{4} = \frac{ \beta(2pq-p-q) +2(\eta +\gamma) - \sqrt{ 4\eta^2(p-1)(q-1) + \beta^2 (p-q)^2 }}{2}$.\\
		\noindent
		$\mathbf{Case 2.}$ If $\eta =0$ and $ \alpha +	\eta  \neq 0$, then $\mathbb{K}$ is as given below:
		\begin{equation*}
			\mathbb{K} =
			\begin{pmatrix}
				\kappa_{11}&\alpha \sqrt{\phi(pq) \phi(1)}&\alpha \sqrt{\phi(p) \phi(1)} &\alpha \sqrt{\phi(q) \phi(1)} \\
				\alpha \sqrt{\phi(pq) \phi(1)}&\kappa_{22}&\alpha \sqrt{\phi(pq) \phi(p)} &\alpha \sqrt{\phi(pq) \phi(q)} \\
				\alpha \sqrt{\phi(p) \phi(1)} & \alpha \sqrt{\phi(pq) \phi(p)} & \kappa_{33} & 0\\
				\alpha \sqrt{\phi(q) \phi(1)}&\alpha \sqrt{\phi(pq) \phi(q)}&0 &\kappa_{44}
			\end{pmatrix}
		\end{equation*}	 
		where  $\kappa_{ii}$ = $\alpha( \phi\left(\frac{pq}{d_{i}} \right) -1) +\beta(\phi(\frac{pq}{d_i}) + \sum\limits_{d_{i} \sim d_{j}} \phi\left(\frac{pq}{d_{j}} \right) -1)+\gamma +\eta\phi\left(\frac{pq}{d_{i}} \right)$. \\
		
		Then the characteristic  polynomial of $\mathbb{K}$ is \\
		\begin{equation*}
			\psi (\mathbb{K}; \lambda) =	\vert\mathbb{K} -\lambda I\vert =
			\begin{vmatrix}
				\kappa_{11} - \lambda&\alpha \sqrt{\phi(pq) \phi(1)}&\alpha \sqrt{\phi(p) \phi(1)} &\alpha \sqrt{\phi(q) \phi(1)} \\
				\alpha \sqrt{\phi(pq) \phi(1)}&\kappa_{22} - \lambda&\alpha \sqrt{\phi(pq) \phi(p)} &\alpha \sqrt{\phi(pq) \phi(q)} \\
				\alpha \sqrt{\phi(p) \phi(1)} & \alpha \sqrt{\phi(pq) \phi(p)} & \kappa_{33}- \lambda & 0\\
				\alpha \sqrt{\phi(q) \phi(1)}&\alpha \sqrt{\phi(pq) \phi(q)}&0 &\kappa_{44} -\lambda
			\end{vmatrix} 
		\end{equation*}
		Let us use $C_{1}$, $C_{2}$, $C_{3}$, $C_{4}$ and $R_{1}$, $R_{2}$, $R_{3}$, $R_{4}$ to represent the columns and rows of $\psi (\mathbb{K}; \lambda)$ respectively. Taking common $\sqrt{\phi(1)}$, $\sqrt{\phi(pq)}$, $\sqrt{\phi(p))}$, $\sqrt{\phi(q)}$ from $R_{1}$, $R_{2}$, $R_{3}$, $R_{4}$ and from  $C_{1}$, $C_{2}$, $C_{3}$, $C_{4}$ correspondingly, we have\\ 
		\begin{equation*}\label{symmetricmatrix}
			\psi (\mathbb{K} ; \lambda)  = \phi(1)\phi(pq)\phi(p)\phi(q)
			\begin{vmatrix}
				\frac{\kappa_{11} - \lambda}{\phi(1)}&\alpha &\alpha&\alpha \\
				\alpha &\frac{\kappa_{22} - \lambda}{\phi(pq)}&\alpha &\alpha  \\
				\alpha  & \alpha & \frac{\kappa_{33}- \lambda}{\phi(p)} & 0\\
				\alpha&\alpha &0 &\frac{\kappa_{44} -\lambda}{\phi(q)}
			\end{vmatrix} 
		\end{equation*} 
		Performing the column  operation $C_{1} \rightarrow C_{1} -C_{2}$ and $C_{4} \rightarrow C_{4} -C_{3}$, we get
		\begin{equation*}\label{symmetricmatrix}
			\psi (\mathbb{K} ; \lambda)  = \phi(1)\phi(pq)\phi(p)\phi(q)
			\begin{vmatrix}
				\frac{\kappa_{11} - \lambda}{\phi(1)} -\alpha&\alpha &\alpha&0 \\
				\alpha- \frac{\kappa_{2} - \lambda}{\phi(pq)} &\frac{\kappa_{22} - \lambda}{\phi(pq)}&\alpha &0  \\
				0  & \alpha & \frac{\kappa_{33}- \lambda}{\phi(p)} & - \frac{\kappa_{3}- \lambda}{\phi(p)}\\
				0&\alpha &0 &\frac{\kappa_{44} -\lambda}{\phi(q)}
			\end{vmatrix} 
		\end{equation*} 
		Performing the row  operation $R_{1} \rightarrow R_{1} -R_{2}$ and $R_{4} \rightarrow R_{4} -R_{3}$, we get
		\begin{equation*}\label{symmetricmatrix}
			\psi (\mathbb{K} ; \lambda)  = \phi(1)\phi(pq)\phi(p)\phi(q)
			\begin{vmatrix}
				\frac{\kappa_{11} - \lambda}{\phi(1)} +  \frac{\kappa_{22} - \lambda}{\phi(pq)}  &  \alpha - \frac{\kappa_{22} - \lambda}{\phi(pq)} &  0  &0 \\
				\alpha- \frac{\kappa_{22} - \lambda}{\phi(pq)}  &  \frac{\kappa_{22} - \lambda}{\phi(pq)}  &  \alpha &  0  \\
				0  &  \alpha &  \frac{\kappa_{33}- \lambda}{\phi(p)}  &  - \frac{\kappa_{33}- \lambda}{\phi(p)}\\
				0  &  0  &  - \frac{\kappa_{33}- \lambda}{\phi(p)}   & \frac{\kappa_{44} -\lambda}{\phi(q)} +  \frac{\kappa_{33}- \lambda}{\phi(p)}
			\end{vmatrix} 
		\end{equation*}
		
		Then we get the required result in this case.
		\\
		\noindent
		$\mathbf{Case 3.}$ If $\eta =0$ and $ \alpha +	\eta  = 0 $,  we obtain $\alpha = 0 $. Therefore, $U(\mathscr{P}(\mathbb{Z}_n))$ is undefined.\\
		$\mathbf{Case 4.}$ If $\eta \neq 0$ and $\eta + \alpha \neq 0 $, then 
		\begin{equation*}
			\kappa_{ij} =\begin{cases}
				( \alpha +	\eta )\sqrt{\phi(\frac{pq}{d_i}){\phi(\frac{pq}{d_j})}}             & \text{if~~}  d_i |d_j \text{~or~}  d_j |d_i   \\ 
				\eta\sqrt{\phi(\frac{pq}{d_i}){\phi(\frac{pq}{d_j})}}                  & \text{otherwise~~~~~~~}  
			\end{cases}  \text{ i $\neq$ j}
		\end{equation*}
		\begin{equation*}
			\kappa_{ii} =\begin{cases}
				\alpha \left(\phi(\frac{pq}{d_{i}}) - 1 \right) + \beta \left(\phi(\frac{pq}{d_i})-1 +  \sum\limits_{d_{i} \sim d_{j}} \phi(\frac{pq}{d_{j}}) \right) +\gamma +\eta \phi\left(\frac{pq}{d_{i}}\right)                       
			\end{cases}
		\end{equation*}
	\end{proof}
	
	\begin{re} 
		{\rm   By assuming $(\alpha, \beta,\gamma, \eta) =  (1,1,0,0)$, the result (\cite{banerjee2019signless} , Theorem 5.4) can be obtained as a specific instance of Theorem \ref{theorem3.3}}   
	\end{re} 
	\section {Universal Adjacency Spectrum of $\overline{\mathscr{P}(\mathbb{Z}_{n})}$}\label{section4}

	The universal adjacency spectrum of complement of a power graph can be obtained as a special case of that of the power graph.

	\begin{theorem}\label{theorem4.4}	According to the definition of notations in  Theorem \ref{theorem3.2} , the  eigenvalues of  $U(\overline{\mathscr{P}({Z}_n)})$ and corresponding eigenvectors are as given below:
		
		\begin{enumerate}
			\item The eigenvalue $\Lambda_{i}$ =	$\beta \left( n -\phi(\frac{n}{d_i}) - \sum\limits_{d_{i} \sim d_{j}} \phi(\frac{n}{d_{j}}) \right) +\gamma $     with multiplicity $\left(\phi(\frac{n}{d_{i}}) -1 \right)$  with the associated eigenvector $\textbf{X}_{i,r}$ = $(\textbf{x}_1, \textbf{x}_2 ,\ldots,\textbf{x}_t) ^{T}$, where 
			\begin{equation*}
				\textbf{x}_k =\begin{cases}
					\textbf{e}_{\phi(\frac{n}{d_i}),r} ^{T},             & \text{if~~}  k=i,\\
					0            ,     & \text{else},
				\end{cases} \text{~~where i= 1,2,\ldots,
					t~~} \text{ and r= 2,3,\ldots, $\phi \left(\frac{n}{d_i} \right)$ ~~}
			\end{equation*}
			
			\item And the rest of  $t$ eigenvalues of $U(\overline{\mathscr{P}({Z}_n)})$ are the eigenvalues of  $\mathbb{K}$  with associated eigenvectors   $(\nu_{i,1} \sqrt\frac{\phi(\frac{n}{d_t})}{\phi(\frac{n}{d_1})} \textbf{j} _{\phi(\frac{n}{d_1})} ,$ $  \nu_{i,2} \sqrt\frac{\phi(\frac{n}{d_t})}{\phi(\frac{n}{d_2})} \textbf{j} _{\phi(\frac{n}{d_2})}, $ $ \ldots, $ $ \nu_{i,t-1} \sqrt\frac{\phi(\frac{n}{d_t})}{\phi(\frac{n}{d_{t-1}})} \textbf{j} _{\phi(\frac{n}{d_{t-1}})}, $ $ \nu_{i,t} \textbf{j}_{\phi(\frac{n}{d_t})}  ) ^{T} $ 
			where $\nu_{i}$ = $( \nu_{i,1} , \nu_{i,2},\ldots, \nu_{i,t} )^{T} $ and $(\lambda_{i} ,\nu_{i})$ is an eigenpair of $\mathbb{K}$,  
			where  
			\begin{equation*}
				\kappa_{ij} =\begin{cases}
					( \alpha +	\eta )\sqrt{\phi(\frac{n}{d_i}){\phi(\frac{n}{d_j})}}             & \text{if neither~~}  d_i |d_j \text{~nor~}  d_j |d_i \\ 
					\eta\sqrt{\phi(\frac{n}{d_i}){\phi(\frac{n}{d_j})}}                  & \text{otherwise~~~~~~~}  
				\end{cases} \text{i $\neq$ j}
			\end{equation*}
			\\and \\
			\begin{equation*}
				\kappa_{ii} =\begin{cases}
					\beta \left( n - \phi(\frac{n}{d_i})  +\sum\limits_{d_{i} \sim d_{j}} \phi(\frac{n}{d_{j}}) \right) +\gamma +\eta \phi\left(\frac{n}{d_{i}}\right)                      
				\end{cases}
			\end{equation*}
		\end{enumerate}
	\end{theorem}
	
	\begin{proof}
		The adjacency and degree diagonal matrices of  $\overline{\mathscr{P}(\mathbb{Z}_n)}$ can be expressed as $A(\overline{\mathscr{P}(\mathbb{Z}_n)}) =   J_n-I_n-A(\mathscr{P}(\mathbb{Z}_n)) $ and $D(\overline{\mathscr{P}(\mathbb{Z}_n)}) = (n-1)I_n - D(\mathscr{P}(\mathbb{Z}_n))$. Then 
		$U(\overline{\mathscr{P}(\mathbb{Z}_n)})= \alpha A (\overline{\mathscr{P}(\mathbb{Z}_n)})+ \beta D(\overline{\mathscr{P}(\mathbb{Z}_n)})+ \gamma I_n + \eta J_n= - \alpha A(\mathscr{P}(\mathbb{Z}_n)) - \beta D(\mathscr{P}(\mathbb{Z}_n)) + (\gamma +\beta(n-1)-\alpha)I_n +( \alpha + \eta )J_n$. Now we have the required result by substituting  $\alpha= -\alpha, \beta=-\beta, \gamma= \gamma +\beta(n-1)-\alpha$ and $\eta= \alpha +\eta $ in Theorem \ref{theorem3.2}.
		
	\end{proof}
	
	The corollary below gives adjacency spectrum and corresponding eigenvectors of $\overline{\mathscr{P}(\mathbb{Z}_{pq})}$.
	
	\begin{corollary}\label{example1}
		For distinct primes $p$and  $q$,   the  eigenvalues of  $\overline{\mathscr{P}(\mathbb{Z}_{pq})}$ and corresponding eigenvectors are as given below:
		
		\begin{enumerate}
			
			\item The eigenvalue $0$ with multiplicity $pq-2$, and corresponding eigenvectors $\textbf{X}_{i,r}$ = $(\textbf{x}_1, \textbf{x}_2 , \textbf{x}_3,\textbf{x}_4) ^{T}$, where 
			\begin{equation*}
				\textbf{x}_k =\begin{cases}
					\textbf{e}_{\phi(\frac{pq}{d_i}),r} ^{T},             & \text{if~~}  k=i,\\
					0            ,     & \text{otherwise},
				\end{cases} \text{~~where i= 1,2,3,4~~} \text{ and r= 2,3,\ldots, $\phi \left(\frac{pq}{d_i} \right)$ ~~}
			\end{equation*}
			and $\left(\sqrt{(q-1)},\underbrace{ 0,0,\ldots,0}_{(pq-1)} \right)$, $\left(0,\underbrace{\frac{1}{\sqrt{(p-1)}},\ldots,\frac{1}{\sqrt{(p-1)}}}_{\phi(pq)}, \underbrace{0, 0,\ldots, 0}_{(pq-\phi(pq)-1)}\right) $

			\item And the eigenvalues $\sqrt{(p-1)(q-1)},$ $-\sqrt{(p-1)(q-1)} $ with corresponding  eigenvectors $\bigg(\underbrace{0,\ldots, 0}_{\phi(pq)+1}, $ $ \underbrace{\sqrt{\frac{q-1}{p-1}},\ldots,\sqrt{\frac{q-1}{p-1}}}_{(p-1)},$ $ \underbrace{1,1,\ldots,1}_{(q-1)} \bigg),$  $\bigg(\underbrace{0,\ldots, 0}_{\phi(pq)+1}, $ $ \underbrace{\sqrt{\frac{q-1}{p-1}},\ldots,\sqrt{\frac{q-1}{p-1}}}_{(p-1)}, $ $ \underbrace{-1,-1,\ldots,-1}_{(q-1)} \bigg)$ respectively.
		\end{enumerate}
	\end{corollary}
	\begin{proof}Here we obtain the eigenvalues and eigenvectors of $A(\overline{\mathscr{P}(\mathbb{Z}_{pq})})$ by applying Theorem \ref{theorem4.4} and  assuming $(\alpha, \beta,\gamma,\eta) = (1,0,0,0)$. We take $d_1= pq, d_2 = 1, d_3 = q, d_4= p$. From  the first part of Theorem \ref{theorem4.4} the eigenvalue of $A(\overline{\mathscr{P}(\mathbb{Z}_{pq})})$  is $0$ with multiplicity  
		$pq-4$.	By the second part of Theorem \ref{theorem4.4},  the other eigenvalues of $A(\overline{\mathscr{P}(\mathbb{Z}_{pq})})$  are obtained from the following matrix:
		\begin{equation*}\label{symmetricmatrix}
			\mathbb{K}  =
			\begin{pmatrix}
				0&0&0 &0 \\
				0&0&0  &0 \\
				0 & 0 & 0 & \sqrt{(p-1)(q-1)}\\
				0&0&\sqrt{(p-1)(q-1)} &0
			\end{pmatrix} 
		\end{equation*}
		Now the spectrum of $\mathbb{K}$ comprise of $0$,$0$, $\sqrt{(p-1)(q-1)}$ and $- \sqrt{(p-1)(q-1)}$ with the corresponding eigenvectors $(1,0,0,0)$, $(0,1,0,0)$, $(0,0,1,1)$ and $(0 ,0,1,-1)$ respectively. The eigenvectors corresponding to these eigenvalues in $A(\overline{\mathscr{P}(\mathbb{Z}_{pq})})$ can be obtained by Theorem \ref{theorem4.4}.
	\end{proof}
	
	For any finite group $G$, we have $U(\overline{\mathscr{P}(G)})= \alpha A(\overline{\mathscr{P}(G})+ \beta D(\overline{\mathscr{P}(G)})+ \gamma I + \eta J$.
	If $\mathbf{\eta= 0}$, then this matrix gives adjacency, Laplacian, signless Laplacian and many other spectra of $\overline{\mathscr{P}(G)}$.  We determine the entire spectrum of $U(\overline{\mathscr{P}(\mathbb{Z}_{pq})})$ in the theorem below  when $\mathbf{\eta= 0}$.
	
	\begin{theorem}
		If $\eta = 0$,  the  eigenvalues of  $U(\overline{\mathscr{P}(\mathbb{Z}_{pq})})$ and corresponding eigenvectors are as given below:
		\begin{enumerate}
			
			\item  The eigenvalue $\gamma$ with multiplicity $\phi(pq)+1$, and corresponding eigenvectors $\textbf{X}_{i,r}$ = $(\textbf{x}_1, \textbf{x}_2 , \textbf{x}_3,\textbf{x}_4) ^{T}$, where 
			\begin{equation*}
				\textbf{x}_k =\begin{cases}
					\textbf{e}_{\phi(pq),r} ^{T},             & \text{if~~}  k=i,\\
					0            ,     & \text{otherwise},
				\end{cases} \text{~~where i= 2~~} \text{ and r= 2,3,\ldots, $\phi \left(pq \right)$ ~~}
			\end{equation*}
			and $(\sqrt{q-1},\underbrace{0,0,\ldots,0}_{(pq-1)}), \left(0,\underbrace{\frac{1}{\sqrt{(p-1)}},\ldots,\frac{1}{\sqrt{(p-1)}}}_{\phi(pq)}, \underbrace{0, 0,\ldots, 0}_{(pq-\phi(pq)-1)}\right) $ 
			
			\item The eigenvalue $\beta(q-1)+ \gamma $ with multiplicity $(p-2)$, and  corresponding eigenvectors $\textbf{X}_{i,r}$ = $(\textbf{x}_1, \textbf{x}_2 , \textbf{x}_3,\textbf{x}_4) ^{T}$, where 
			\begin{equation*}
				\textbf{x}_k =\begin{cases}
					\textbf{e}_{\phi(\frac{pq}{q}),r} ^{T},             & \text{if~~}  k=i,\\
					0            ,     & \text{otherwise},
				\end{cases} \text{~~where i= 3~~} \text{ and r= 2,3,\ldots, $\phi(p )$ ~~}
			\end{equation*}
			\item The eigenvalue $\beta(p-1)+ \gamma$ with multiplicity $(q-2)$, and has corresponding eigenvectors  $\textbf{X}_{i,r}$ = $(\textbf{x}_1, \textbf{x}_2 , \textbf{x}_3,\textbf{x}_4) ^{T}$, where 
			\begin{equation*}
				\textbf{x}_k =\begin{cases}
					\textbf{e}_{\phi(\frac{pq}{p}),r} ^{T},             & \text{if~~}  k=i,\\
					0            ,     & \text{otherwise},
				\end{cases} \text{~~where i= 4~~} \text{ and r= 2,3,\ldots, $\phi (q )$ ~~}
			\end{equation*}
			\item The eigenvalues $\frac{\beta(p+q-2)+2\gamma}{2}+$ $\frac {\sqrt{(\beta(p-q))^2 +4 \alpha^2(p-1)(q-1)}}{2}$ and $\frac{\beta(p+q-2)+2\gamma}{2}-$ $\frac {\sqrt{(\beta(p-q))^2 +4 \alpha^2(p-1)(q-1)}}{2}$ with corresponding eigenvectors  $\big(0,$  $0,$ $\frac{2\alpha \sqrt{(p-1)(q-1)}}{\beta(p-q)-\sqrt{(\beta(p-q))^2 +4 \alpha^2(p-1)(q-1)}},$ $1\big)$ and  $\big(0,$ $0,$ $\frac{2\alpha \sqrt{(p-1)(q-1)}}{\beta(p-q)+\sqrt{(\beta(p-q))^2 +4 \alpha^2(p-1)(q-1)}},$ $1\big)$ respectively.
		\end{enumerate}
		
	\end{theorem}
	\begin{proof}
		We  take $ H_{pq} \cup H_1   \cup H_{q} \cup H_p$ as a partition of  $V(\overline{\mathscr{P}(\mathbb{Z}_{pq})})$.  Then by the first part of Theorem \ref{theorem4.4} we get the  eigenvalues of   $U(\overline{\mathscr{P}(\mathbb{Z}_{pq})})$ as $\gamma$,  $\beta(q-1)+ \gamma $, $\beta(p-1)+ \gamma$ with multiplicity  $\phi(pq)+1$,  $(p-2)$, $(q-2)$ respectively. Now
		\begin{equation*}
			\mathbb{K} = \begin{bmatrix} 
				\gamma & 0 &0 &0\\
				0& \gamma & 0 &0\\
				0&0  & (\beta (q-1 )+\gamma)& \alpha\sqrt{\phi(p)\phi(q)} \\
				0&0& \alpha \sqrt{\phi(p)\phi(q)}& (\beta(p-1)+\gamma ) 		 
			\end{bmatrix} 
		\end{equation*}
		
		The spectrum of  $\mathbb{K}$ comprises of  the eigenvalues $\gamma,$ $\gamma,$ $\frac{\beta(p+q-2)+2\gamma}{2}-$ $\frac {\sqrt{(\beta(p-q))^2 +4 \alpha^2(p-1)(q-1)}}{2},$ $\frac{\beta(p+q-2)+2\gamma}{2}+$ $\frac {\sqrt{(\beta(p-q))^2 +4 \alpha^2(p-1)(q-1)}}{2},$ and the corresponding eigenvectors  are $(0,1,0,0),$ $(0,1,0,0),$ $\big(0,0,$  $\frac{2\alpha \sqrt{(p-1)(q-1)}}{\beta(p-q)-\sqrt{(\beta(p-q))^2 +4 \alpha^2(p-1)(q-1)}}, $  $1\big), $ $\left(0,0,\frac{2\alpha \sqrt{(p-1)(q-1)}}{\beta(p-q)+\sqrt{(\beta(p-q))^2 +4 \alpha^2(p-1)(q-1)}},1\right)$ respectively.

	\end{proof}

	
	\section{Universal Adjacency Spectrum of Power Graph of Dihedral Group  }\label{section5}
	
	We recall that is the dihedral group ${D}_{n}$ = $<a,b>$, the generators  satisfy  $o(a)$ = $n$,  $o(b)$ = $2$, $ba = a^{-1} b $. The order of group ${D}_{n}$  is  $2n$. So ${D}_n = \{e , a, a^2,a^3,\ldots,a^{n-1}, b,ba,ba^2,ba^3,\ldots,ba^{n-1}\}.$ The elements $e,a,a^2,\ldots,a^{n-1}$ are rotations and $b,ba,ba^2, ba^3,\ldots,ba^{n-1}$ are   reflections. The order of each reflection is $2$. The cyclic subgroup $<a>$  of ${D}_{n}$ is isomorphic to $\mathbb{Z}_{n}$. As before we take the distinct positive divisors of $n$  as $d_{1}, d_{2},\ldots ,d_{t}$, and $\Omega_n$ as the graph considered in Section \ref{2}. Let $\Omega_n^{\prime}$ be the graph constructed from $\Omega_n$ by introducing a new vertex $R$ and making $R$ adjacent with $d_i$ if and only if $d_i = n$. 
	
	Let $S_{d_{i}}$ = $\{ a^x \in {D}_n : \gcd(x,n)= d_{i} \}$, for $1 \leq i \leq t$, and $S_{R}$ = $\{b,ba,ba^2,ba^3,\ldots,ba^{n-1}\}.$ Then $S_{d_{1}} \cup S_{d_{2}}  \ldots \cup S_{d_{t}} \cup  S_{R}$ forms a partition of $V(\mathscr{P}({D}_n))$.
	Proceeding in the similar manner as in the proof of Proposition \ref{lemma2} we get that  a vertex of $S_{d_i}$ is adjacent to a vertex of  $S_{d_j}$ if and only if either $d_{i}|d_{j}$ or $d_{j}|d_{i}$.  Also $S_{d_i}$ induces the complete graph $K_{\phi(\frac{n}{d_i})}$, $1 \leq i \leq t$, and $S_{R}$ induces an independent set, that is, $\overline{K}_n$. Moreover, from the definition of power graph, a vertex of  $S_{d_i}$ is adjacent to a vertex of $S_{R}$ if and only if $d_{i}$ =$n$. Now
	\begin{equation}\label{structure}
		\mathscr{P}({D}_n) \cong \bigvee_{\Omega_n^{\prime}}\{\mathscr{P}(S_{d_1}), \mathscr{P}(S_{d_2}),\ldots, \mathscr{P}(S_{d_t}), \mathscr{P}(S_{R})\}
	\end{equation}
	
	\begin{theorem}\label{theorem5.1}
		The  eigenvalues of  $U(\mathscr{P}({D}_n))$  and corresponding eigenvectors are as given below:
		
		\begin{enumerate}
			
			\item The eigenvalue $\Lambda_{i} = -\alpha + \left( \phi(\frac{n}{d_i})-1 + \sum\limits_{d_{i} \sim d_{j}} \phi(\frac{n}{d_{j}}) \right)\beta +\gamma$ with multiplicity $\left( \phi(\frac{n}{d_{i}}) - 1\right)$ with the associated eigenvector $\textbf{X}_{i,r}$ = $(\textbf{x}_1 , \textbf{x}_2 ,\ldots,\textbf{x}_t,\textbf{x}_R) ^{T}$, where 
			\begin{equation*}
				\textbf{x}_k =\begin{cases}
					\textbf{e}_{\phi (\frac{n}{d_i} ),r} ^{T}  ,             & \text{if~~}  k=i,\\
					0           ,      & \text{otherwise},
				\end{cases} \text{~~where i= 1,2,\ldots,	t~~} \text{ and r= 2,3,\ldots, $\phi \left(\frac{n}{d_i} \right)$ }
			\end{equation*}
			
			\item The eigenvalue $\Lambda_{R}$ = $\beta + \gamma$ with multiplicity  $(n-1)$ with the associated eigenvector $\textbf{X}_{i,r}$ = $(\textbf{x}_1, \textbf{x}_2,\ldots,\textbf{x}_t,\textbf{x}_R) ^{T}$, where 
			\begin{equation*}
				\textbf{x}_k =\begin{cases}
					\textbf{e}_{\phi (\frac{n}{d_i} ),r} ^{T},             & \text{if~~}  k=R,\\
					0,     & \text{otherwise},
				\end{cases}
				\text{ and r = 2,3, \ldots,n }
			\end{equation*}
			
			\item  And the rest of  $t+1$ eigenvalues of $U(\mathscr{P}({D}_n))$ are the eigenvalues of $\mathbb{K}$  with associated eigenvectors   $(\nu_{i,1} \sqrt\frac{n}{\phi(\frac{n}{d_1})} \textbf{j} _{\phi(\frac{n}{d_{1}})} , \nu_{i,2} \sqrt\frac{n}{\phi(\frac{n}{d_2})} \textbf{j} _{\phi(\frac{n}{d_{2}})},\ldots,$ $\nu_{i,t} \sqrt\frac{n}{\phi(\frac{n}{d_{t}})} \textbf{j} _{\phi(\frac{n}{d_{t}})}  , \nu_{i,R} \textbf{j}_ n ) ^{T} $ 
			where $\nu_{i}$ = $( \nu_{i,1} , \nu_{i,2},\ldots, \nu_{i,t},\nu_{i,R} )^{T} $ and $(\lambda_{i} ,\nu_{i})$ is an eigenpair of $\mathbb{K}$  which is given below :

			\begin{equation}\label{symmetricmatrix15}
				\mathbb{K} =
				\begin{pmatrix}
					\kappa_{11}&\theta_{1,2} \sqrt{\phi(\frac{n}{d_1}){\phi(\frac{n}{d_2})}}&\cdots &\theta_{1,t}  \sqrt{\phi(\frac{n}{d_1}){\phi(\frac{n}{d_t})}} &\theta_{1,R}\sqrt{n \phi(\frac{n}{d_1})} \\
					\theta_{2,1} \sqrt{\phi(\frac{n}{d_2}){\phi(\frac{n}{d_1})}}&\kappa_{22}&\cdots &\theta_{2,t} \sqrt{\phi(\frac{n}{d_2}){\phi(\frac{n}{d_k})}} &\theta_{2,R}\sqrt{n\phi(\frac{n}{d_2})}  \\
					\vdots & \vdots & \ddots & \vdots\\
					\theta_{t,1} \sqrt{\phi(\frac{n}{d_t}){\phi(\frac{n}{d_1})}}&\theta_{t,2} \sqrt{\phi(\frac{n}{d_t}){\phi(\frac{n}{d_2})}}&\cdots &\kappa_{tt}& \theta_{t,R}\sqrt{n\phi(\frac{n}{d_t})} \\
					\theta_{R,1}\sqrt{ n\phi(\frac{n}{d_1})}& \theta_{R,2}\sqrt{n \phi(\frac{n}{d_2})}& \cdots&\theta_{R,t}\sqrt{n\phi(\frac{n}{d_t})}& \kappa_{RR}
				\end{pmatrix}
			\end{equation}
			
			where
			\begin{equation*}
				\kappa_{ii} =\begin{cases} \alpha \left(\ \phi(\frac{n}{d_{i}} ) -1 \right) +\beta(\phi(\frac{n}{d_i})-1+ 
					\sum\limits_{d_{i} \sim d_{j}} \phi\left(\frac{n}{d_{j}} \right) )+\gamma + \phi\left(\frac{n}{d_{i}} \right)\eta,  & \text{when $d_i$ $\neq$ $n$ }\\
					(2n-1)\beta + \gamma + \eta, & \text{ when $d_i$ = $n$} \\
					\beta+ \gamma +\eta n, &\text {i= R}
				\end{cases}
			\end{equation*}
			
			\begin{equation*}
				\theta_{i,j} =\begin{cases}
					\alpha +	\eta,             & \text{if~~}  d_i |d_j \text{~or~} d_j |d_i,  \\ 
					\eta,      & \text{else}  ,
				\end{cases}   1\leq i,j \leq t; 
			\end{equation*} 
			
			\begin{equation*}
				\theta_{R,i} = \theta_{i,R}= \begin{cases}
					\alpha +	\eta,             & \text{if~~}  d_i = n, \\ 
					\eta,       & \text{else},
				\end{cases}   1\leq i \leq t;
			\end{equation*} 
		\end{enumerate}
	\end{theorem}

	\begin{proof}
		The eigenvalues of $\mathscr{P}(S_{d_i})$ are $(\phi(\frac{n}{d_{i}})-1)$ and $(-1)$ with the multiplicities $1$ and $(\phi(\frac{n}{d_{i}})-1)$ respectively.  $\mathscr{P}(S_R)$ has the eigenvalue $0$ with multipicity $n$. The result then follows from Theorem \ref{theorem3.1}.		
	\end{proof}

    \begin{re}
	{ \rm By assuming $(\alpha, \beta, \gamma, \eta)= ((1-\alpha), \alpha,0,0)$, the Theorem \ref{theorem5.1} becomes the solution for the Problem $3$ stated in \cite{rather2022alpha}.}
 	\end{re}
	
	\begin{re}\label{remark5.1}
		{\rm   By assuming $n = pq$ and $(\alpha, \beta,\gamma, \eta) =  (-1,1,0,0)$, the result (\cite{chattopadhyay2015laplacian}, Corollary 3.5) can be obtained as a specific instance of Theorem \ref{theorem5.1}}   
	\end{re}
	
	We explain Remark \ref{remark5.1} in the example below.
	
	\begin{example}
		{\rm Here we find Laplacian eigenvalues and eigenvectors of $D_{15}$ applying Theorem \ref{theorem5.1} and assuming $(\alpha, \beta,\gamma,\eta) = (-1,1,0,0)$.
			The positive divisors of 15 are 1,3,5 and 15. Now $S_1,S_3,S_5,S_{15}$ and $S_R$ induce $K_8, K_4, K_2,K_1$ and $K_{15}$ respectively in $\mathscr{P}({D}_{15})$. The graph $\Omega ' _{15}$ is given in Figure $1(a)$. So power graph of $D_{15}$  looks like as given in Figure $1(b)$.
			\begin{figure}[h!]
				\centering
				
				\subfigure[]{\includegraphics[width=0.24\textwidth]{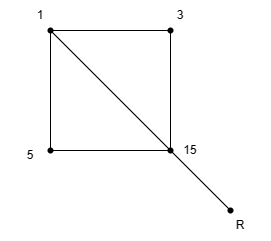}} \hspace{2cm}
				\subfigure[]{\includegraphics[width=0.22\textwidth]{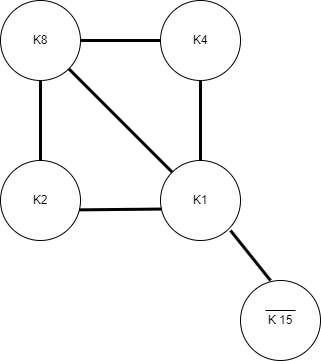}} 
				\caption{(a)$\Omega' _{15}$ \hspace{2cm} (b) ${\mathscr{P}(D_{15})}$ }
				
				\label{fig:foobar}
			\end{figure}
			
			From  the first and second parts of Theorem \ref{theorem5.1} the eigenvalues of $L({\mathscr{P}({D}_{15})})$  are 16, 13, 11, 1 with multiplicity 7, 4, 1, 14 respectively. By the third part of Theorem \ref{theorem5.1}, the other  eigenvalues of $L({\mathscr{P}(D_{15})})$ can be obtained  from   the matrix below: 
			\begin{equation*}\label{symmetricmatrix}
				\mathbb{K}  = \scriptsize
				\begin{pmatrix}
					29&-\sqrt{\phi(15)}&-\sqrt{\phi(3)} &-\sqrt{\phi(5)} &-\sqrt{15}\\
					-\sqrt{\phi(15)}&7&-\sqrt{\phi(15)\phi(3)}&-\sqrt{\phi(15)\phi(5)}  &0 \\
					-\sqrt{\phi(3)}  &-\sqrt{\phi(15) \phi(3)} & 9 & 0&0\\
					-\sqrt{\phi(5)}&-\sqrt{\phi(15)\phi(5)}&0 &9&0\\
					-\sqrt{15}&0 &0 &0 &1
				\end{pmatrix} 
			\end{equation*}
			Now the eigenvalues of $\mathbb{K}$ are $30,15,9,1,0$ with the corresponding eigenvectors $\big(\frac{-29}{\sqrt{15}},$ $\sqrt{\frac{8}{15}},$ $\sqrt{\frac{2}{15}}, $ $\sqrt{\frac{4}{15}},$ $1\big),$ $\big(0,$ $\frac{-3}{\sqrt{8}},$ $\frac{1}{\sqrt{2}},$ $1,$ $0\big),$ $\big(0,$ $0,$ $-\sqrt{2},$ $1,$ $0\big),$ $\big(0,$ $-\frac{30}{\sqrt{7}},$ $-\frac{\sqrt{\frac{15}{2}}}{7},$ $-\frac{15}{\sqrt{7}},1\big),$ $\big(\frac{1}{\sqrt{15}},$ $\sqrt{\frac{8}{15}},$ $\sqrt{\frac{2}{15}},$ $\sqrt{\frac{4}{15}},$ $1\big)$ respectively. The eigenvectors of these eigenvalues in $L({\mathscr{P}(D_{15})})$ can be obtained from the third part of Theorem \ref{theorem5.1}.}   
		
	\end{example}
	
	The  theorem below can be proved analogous to Theorem \ref{theorem4.4}.

	\begin{theorem}\label{theorem5.2}
		The  eigenvalues of $U(\overline{\mathscr{P}({D}_n)})$ and corresponding eigenvectors are as given below:
		\begin{enumerate}
			\item  The eigenvalue $\Lambda_{i}$ =	$\beta \left( 2n - \phi(\frac{n}{d_{i}}) -  \sum\limits_{d_{i} \sim d_{j}} \phi(\frac{n}{d_{j}}) \right) +\gamma $     with multiplicity $\left(\phi(\frac{n}{d_{i}}) -1\right)$  with the associated eigenvector $\textbf{X}_{i,r}$ = $(\textbf{x}_1 , \textbf{x}_2 ,...,\textbf{x}_t,\textbf{x}_R) ^{T}$, where 
			\begin{equation*}
				\textbf{x}_k =\begin{cases}
					\textbf{e}_{\phi (\frac{n}{d_i}),r} ^{T}  ,             & \text{if~~}  k=i,\\
					0           ,      & \text{otherwise},
				\end{cases} \text{~~where i= 1,2,\ldots,t~~} \text{ and r= 2,3,\ldots, $\phi \left(\frac{n}{d_i}\right)$ }
			\end{equation*}
			
			\item The eigenvalue $\Lambda_R$ = $-\alpha+(2n-2)\beta + \gamma$ with multiplicity $(n-1)$ with the associated eigenvector $\textbf{X}_{i,r}$ = $(\textbf{x}_1, \textbf{x}_2,...\textbf{x}_t,\textbf{x}_R) ^{T}$, where 
			\begin{equation*}
				\textbf{x}_k =\begin{cases}
					\textbf{e}_{n,r} ^{T},             & \text{if~~}  k=R,\\
					0,     & \text{otherwise},
				\end{cases}
				\text{ and r= 2,3,\ldots,n }
			\end{equation*}
			
			\item And the rest of  $t+1$ eigenvalues of $U(\overline{\mathscr{P}({D}_n)})$ are the eigenvalues of $\mathbb{K}$  with associated eigenvectors    $(\nu_{i,1} \sqrt\frac{n}{\phi(\frac{n}{d_1})} \textbf{j} _{\phi(\frac{n}{d_{1}})}, \nu_{i,2} \sqrt\frac{n}{\phi(\frac{n}{d_2})} \textbf{j} _{\phi(\frac{n}{d_{2}})}, \ldots,$ $ \nu_{i,t} \sqrt\frac{n}{\phi(\frac{n}{d_{t}})} \textbf{j} _{\phi(\frac{n}{d_{t}})}, v_{i,R} \textbf{j}_ n ) ^{T} $ 
			where $\nu_{i}$ = $( \nu_{i,1}, \nu_{i,2},\ldots,\nu_{i,t},\nu_{i,R} )^{T} $ and $(\lambda_{i} ,\nu_{i})$ is an eigenpair of $\mathbb{K}$ which is given below :

			\begin{equation}\label{symmetricmatrix}
				\mathbb{K} = \scriptsize
				\begin{pmatrix}
					\kappa_{1}&\theta_{1,2} \sqrt{\phi(\frac{n}{d_1}){\phi(\frac{n}{d_2})}}&\cdots &\theta_{1,t}  \sqrt{\phi(\frac{n}{d_1}){\phi(\frac{n}{d_t})}} &\theta_{1,R}\sqrt{\phi(\frac{n}{d_1}).n} \\
					\theta_{2,1} \sqrt{\phi(\frac{n}{d_2}){\phi(\frac{n}{d_1})}}&\kappa_{2}&\cdots &\theta_{2,t} \sqrt{\phi(\frac{n}{d_2}){\phi(\frac{n}{d_t})}} &\theta_{2,R}\sqrt{\phi(\frac{n}{d_2}).n}  \\
					\vdots & \vdots & \ddots & \vdots\\
					\theta_{t,1} \sqrt{\phi(\frac{n}{d_t}){\phi(\frac{n}{d_1})}}&\theta_{t,2} \sqrt{\phi(\frac{n}{d_t}){\phi(\frac{n}{d_2})}}&\cdots &\kappa_{t}& \theta_{t,R}\sqrt{\phi(\frac{n}{d_t}).n} \\
					\theta_{R,1}\sqrt{\phi(\frac{n}{d_1}).n}& \theta_{R,2}\sqrt{\phi(\frac{n}{d_2}).n}& \cdots&\theta_{R,t}\sqrt{\phi(\frac{n}{d_t}).n}& \kappa_{R}
				\end{pmatrix}
			\end{equation}
			where
			\begin{equation*}
				\kappa_{i} = \scriptsize \begin{cases}
					\beta\left( 2n -\phi(\frac{n}{d_{i}})- \sum\limits_{d_{i} \sim d_{j}} \phi(\frac{n}{d_{j}}) \right) +\gamma +\eta \phi\left(\frac{n}{d_{i}}\right)               & \text{when $d_{i}$ $\neq$ $n$} ,\\ 
					\gamma + \eta &\text{when $d_i$ = $n$}  \\
					
				\end{cases}  1 \leq i \leq t.
			\end{equation*}

			\begin{equation*}
				\kappa_{R} =  (n-1)\alpha + (2n-2)\beta +\gamma +\eta n 
			\end{equation*}
			
			\begin{equation*} 
				\theta_{i,j} = \scriptsize \begin{cases}
					\eta + \alpha  ,             & \text{if neither~~}  d_i |d_j \text{~nor~}  d_j |d_i  , \\ 
					\eta           ,      & \text{otherwise}, 
				\end{cases}  1\leq i,j \leq t.
			\end{equation*} 
			
			\begin{equation*}
				\theta_{R,i} =  \scriptsize \theta_{i,R}= \begin{cases}
					\eta      ,         & \text{if~~}  d_i = n, \\ 
					\eta +\alpha  ,               & \text{otherwise}, 
				\end{cases}  1\leq i \leq t.
			\end{equation*} 
		\end{enumerate}
		
	\end{theorem}


	\section{Universal Adjacency Spectrum of Power Graph of Generalized Quaternion Group }\label{section6}
	For  $n\geq 2$, the dicyclic group ${Q}_n = <a,b| a^{2n}=e, b^{2}= a^n, ab= ba^{-1}>$ is of order $4n$. If $n$ is a power of $2$, then ${Q}_n$ is known as the generalized quaternion group. Let $T_1$=$\{e, a^n\}$,  $T_2$=$\{a,a^2,\ldots,a^{n-1},a^{n+1},\ldots,a^{2n-1}\}$,   and for $3 \leq i \leq n+2 $, $T_i$=$\{a^{i-3}b, a^{n+i-3}b\}$. So,  $ T_1\cup T_2\cup T_3\cup T_4 \ldots\cup T_{n+2}$ is a partition of $V(\mathscr{P}({Q}_n))$. We see that $T_2$ induces the  complete graph  $K_{2n-2 }$ and each $T_i$ induces $K_2$, for $1 \leq i \leq n+2~ (i \neq 2) $. Let $T$ be the star graph $K_{1,n+1}$ with $V(K_{1,n+1}) = \{1,2,\ldots,n+2\}$ and where  $1$ is non- pendent  vertex. Then   
	$$\mathscr{P}(Q_n) = \bigvee_T \{{ \mathscr{P}(T_1), \mathscr{P}(T_2),\ldots, \mathscr{P}(T_{n+2})}\}$$

	Since the eigenvalues of  $K_m$ are $(-1)^{m-1}, (m-1)$, and those of $K_{1,m}$ are $-\sqrt{m},\sqrt{m},(0)^{m-1}$,  the  result below follows from Theorem \ref{theorem3.1},
	\begin{theorem}\label{theorem6.1}
		The  eigenvalues of $U(\mathscr{P}({Q}_n))$ and corresponding eigenvectors are as given below:
		\begin{enumerate}
			\item The eigenvalue $ -\alpha+ (4n-1)\beta +\gamma $ with the associated eigenvector $(1,-1,\underbrace{0,0,\ldots,0,0}_{4n-2 ~ times})$. \\
			\item The eigenvalue $-\alpha+(2n-1)\beta+\gamma$ with multiplicity $(2n-3)$ with the associated eigenvector $X_i = (0,0,1,x_{i2},x_{i3},\ldots,x_{i(2n-2)},\underbrace{0,0,\ldots,0}_{2n ~times})$, where
			\begin{equation*}
				x_{ij}= \begin{cases}
					-1 & \text{if~~} i=j\\
					0 & \text {otherwise.}
				\end{cases} 	\text{~~~i= 2,\ldots,$(2n-2)$}
			\end{equation*}
			
			\item  The eigenvalue $-\alpha + 3\beta + \gamma$ with multiplicity $n$ with the associated eigenvector \break $\textbf{Y}_{i}$ = $(\textbf{y}_1 , \textbf{y}_2 ,\ldots,\textbf{y}_{(n+1)},\textbf{y}_{(n+2)}) ^{T}$, where $\textbf{y}_1, \textbf{y}_2$ are $0$ vectors of dimension $2$, $(2n-2)$ respectively and
			\begin{equation*}
				\textbf{y}_l =\begin{cases}
					\textbf{e}_{2,2} ^{T}  ,             & \text{if~~}  l=i,\\
					0           ,      & \text{otherwise},
				\end{cases} \text{~~where i= 3,2...,	(n+2)~~}
			\end{equation*}
			
			\item The rest of $(n+2)$ eigenvalues of $U(\mathscr{P}({Q}_n))$ are the eigenvalues of   $\mathbb{K}$ with the corresponding eigenvectors   $(\nu_{i,1} \textbf{j} _{2} ,$ $ \nu_{i,2} \sqrt\frac{2}{2n-2} \textbf{j} _{2n-2}, $ $ \nu_{i,3} \textbf{j} _{2},$ $\ldots, $ $ v_{i, n+1} \textbf{j} _{2}  ,$ $ \nu_{i,n+2} \textbf{j}_{2}  ) ^{T} $ where $\nu_{i}$ = $( \nu_{i,1} , \nu_{i,2},... ,\nu_{i,n+2} )^{T} $  and $(\lambda_i, \nu_i)$ is an eigenpair of $\mathbb{K}$. The columns and rows of $\mathbb{K}$ are  indexed by the vertices of $T$ in the following order  $1, 2,\ldots, n + 2$, where
			\begin{equation*}\label{symmetricmatrix}
				\mathbb{K} =\scriptsize
				\begin{pmatrix}
					
					\kappa_1&(\alpha+\eta)\sqrt{2(2n-2)} &(\alpha+\eta)\sqrt{2.2}&(\alpha+\eta)\sqrt{2.2}&\cdots&(\alpha+\eta)\sqrt{2.2}\\
					(\alpha+\eta)\sqrt{2.(2n-2)}&\kappa_2&\eta \sqrt{2(2n-2)} &\eta\sqrt{2(2n-2)}&\cdots&\eta\sqrt{2(2n-2)} \\
					
					(\alpha+\eta)\sqrt{2.2}&\eta\sqrt{2(2n-2)}&\alpha+3\beta+\gamma+2\eta &\eta\sqrt{2.2}&\cdots&\eta\sqrt{2.2} \\
					(\alpha+\eta)\sqrt{2.2}&\eta\sqrt{2(2n-2)}&\eta \sqrt{2.2} &\alpha+3\beta+\gamma+2\eta&\cdots&\eta\sqrt{2.2}\\
					\vdots&\vdots&\vdots&\vdots&\ddots&\vdots\\	
					(\alpha+\eta)\sqrt{2.2} &\eta\sqrt{2(2n-2)} &\eta\sqrt{2.2}&\eta\sqrt{2.2}&\cdots&\alpha+3\beta+\gamma+2\eta \\
					
				\end{pmatrix}
			\end{equation*}
			
		\end{enumerate}
		where $\kappa_1$ = $\alpha+(4n-1)\beta+ \gamma+ 2\eta$ and $\kappa_2 = (2n-3)\alpha +(2n-1)\beta+ \gamma + (2n-2)\eta$.\\
		The eigenvalues of $\mathbb{K}$ are $\lambda_1$, $\lambda_2$, $\lambda_3$ and $\alpha+3\beta+\gamma$ with multiplicity $1$,$1$,$1$ and $(n-1)$ respectively. 
		
	\end{theorem}

	The next theorem gives universal adjacency spectrum of $\overline{\mathscr{P}({Q}_n)}$, and its proof can be obtained by applying Theorem \ref{theorem4.4}.
	
	\begin{theorem}\label{theorem6.2}
		According to the definition of notations in  Theorem \ref{theorem6.1}, the  eigenvalues of $U(\overline{\mathscr{P}({Q}_n)})$ and corresponding eigenvectors are as given below:
		\begin{enumerate}
			
			\item The eigenvalue $\gamma $ with the associated eigenvector $(1,-1,\underbrace{0,0,...,0,0}_{4n-2 ~ times})$. \\
			\item The eigenvalue $(2n)\beta+\gamma$ with multiplicity $(2n-3)$ with the associated eigenvector $X_i = (0,0,1,x_{i2},x_{i3},...,x_{i(2n-2)},\underbrace{0,0,...,0}_{2n ~times})$, 	where 
			\begin{equation*}
				x_{ij}= \begin{cases}
					-1 & \text{if~~} i=j\\
					0 & \text {otherwise.}
				\end{cases} 	\text{~~~i= 2,...$(2n-2)$}
			\end{equation*}
			
			\item  The eigenvalue $(4n-4)\beta + \gamma$ with multiplicity $n$ and  with the corresponding eigenvector $\textbf{Y}_{i}= (\textbf{y}_1,$ $ \textbf{y}_2,$ $\ldots,$ $\textbf{y}_{(n+1)},$ $\textbf{y}_{(n+2)}) ^{T}$, where  $\textbf{y}_1, \textbf{y}_2$ are $0$ vectors of dimension $2$, $(2n-2)$ respectively and
			\begin{equation*}
				\textbf{y}_l =\begin{cases}
					\textbf{e}_{2,2} ^{T}  ,             & \text{if~~}  l=i,\\
					0           ,      & \text{otherwise},
				\end{cases} \text{~~where i= 3,2...,	(n+2)~~}
			\end{equation*}
			
			\item The rest of  $(n+2)$ eigenvalues of $U(\overline{\mathscr{P}(Q_n)})$ are the eigenvalues of $\mathbb{K}$ with the corresponding eigenvectors   $(\nu_{i,1} \textbf{j} _{2} ,$ $ \nu_{i,2} \sqrt\frac{2}{2n-2} \textbf{j} _{2n-2},$ $ \nu_{i,3} \textbf{j} _{2},$ $ \ldots,\nu_{i, n+1} \textbf{j} _{2},$ $ \nu_{i,n+2} \textbf{j}_{2}  ) ^{T} $ where $\nu_{i}$ = $( \nu_{i,1} , \nu_{i,2},\ldots ,\nu_{i,n+2} )^{T} $  and $(\lambda_i, \nu_i)$ is an eigenpair of $\mathbb{K}$ which is given below:
			
			\begin{equation}\label{symmetricmatrix 17}
				\mathbb{K} =
				\scriptsize
				\begin{pmatrix}
					
					\kappa_1&\eta\sqrt{2(2n-2)} &\eta\sqrt{2.2}&\eta\sqrt{2.2}&\cdots&\eta\sqrt{2.2}\\
					\eta\sqrt{2.(2n-2)}&\kappa_2&(\alpha+\eta) \sqrt{2(2n-2)} &(\alpha+\eta)\sqrt{2(2n-2)}&\cdots&(\alpha+\eta)\sqrt{2(2n-2)} \\
					
					\eta\sqrt{2.2}&(\alpha+\eta)\sqrt{2(2n-2)}&(4n-4)\beta+\gamma+2\eta &(\alpha+\eta)\sqrt{2.2}&\cdots&(\alpha+\eta)\sqrt{2.2} \\
					\vdots&\vdots&\vdots&\ddots&\vdots&\vdots\\	
					\eta\sqrt{2.2} &(\alpha+\eta)\sqrt{2(2n-2)} &(\alpha+\eta)\sqrt{2.2}&(\alpha+\eta)\sqrt{2.2}&\cdots&(4n-4)\beta+\gamma+2\eta \\
					
				\end{pmatrix}
			\end{equation}
			
			where $\kappa_1$ = $ \gamma+ 2\eta$ and $\kappa_2 = (2n)\beta+ \gamma + (2n-2)\eta$.\\
			The eigenvalues of $\mathbb{K}$ are $\lambda_1$, $\lambda_2$, $\lambda_3$ and $-2\alpha+(4n-3)\beta+\gamma$ with multiplicity $1$,$1$,$1$ and $(n-1)$ respectively. 
		\end{enumerate}
		
	\end{theorem}
	
	The example below demonstrates 	Theorem \ref{theorem6.2}	
	\begin{example}	{\rm Here we consider $\overline{\mathscr{P}(Q_{2})}$. So $n=2$. For $\mathscr{P}(Q_{2})$ and $\overline{\mathscr{P}(Q_{2})}$ one may refer Figures $2(a)$ and $2(b)$ respectively.
			\begin{figure}[h!]
				\centering
				
				\subfigure[]{\includegraphics[width=0.22\textwidth]{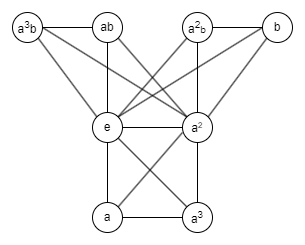}} \hspace{2cm}
				\subfigure[]{\includegraphics[width=0.22\textwidth]{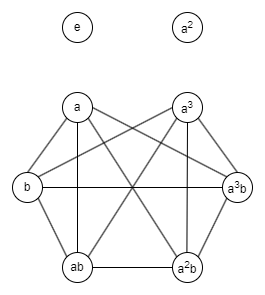}} 
				\caption{(a)$\mathscr{P}(Q_{2})$ \hspace{2cm} (b) $\overline{\mathscr{P}(Q_{2})}$ }
				
				\label{fig:foobar}
			\end{figure}

			 $T_1$=$\{e, a^2\}$,  $T_2$=$\{a,a^{3}\}$, $T_3$=$\{b, a^2 b\}$,  $T_4$=$\{ab, a^3 b\}$. From the first and second parts of the Theorem \ref{theorem6.2}, the eigenvalues are $\gamma$, $4\beta+ \gamma$  with multiplicity $1$, $3$ and the corresponding eigenvectors $(1,-1,0,0,0,0,0,0)$, 
			$(0,0,1,-1,0,0,0,0)$,  $(0,0,0,0,1,-1,0,0)$, $(0,0,0,0,0,0,1,-1)$ respectively.	By the third part of Theorem \ref{theorem6.2}, the  matrix $\mathbb{K}$ is given by:
			\begin{equation*}
				\mathbb{K} =
				\scriptsize
				\begin{pmatrix}
					
					\gamma + 2\eta&2\eta& 2\eta&2\eta\\
					2\eta&4\beta+\gamma+2\eta&2(\alpha+\eta) &2(\alpha+\eta) \\
					
					2\eta&2(\alpha+\eta)&4\beta+\gamma+2\eta &2(\alpha+\eta)\\
					2\eta &2(\alpha+\eta)&2(\alpha+\eta)&4\beta+\gamma+2\eta \\
					
				\end{pmatrix}
			\end{equation*}
			The eigenvalues of $\mathbb{K}$ are 	$2\alpha+2\beta+\gamma+4\eta+2 \sqrt{\alpha^2+ \beta^2+  4 \eta^2 + 2\alpha \beta+ 2\alpha \eta+  2\beta \eta}$, 	$2\alpha+2\beta+\gamma+4\eta -2 \sqrt{\alpha^2+ \beta^2+  4 \eta^2 + 2\alpha \beta+ 2\alpha \eta+  2\beta \eta}$, $-2\alpha+4\beta+\gamma$ with multiplicity $1$, $1$, $2$ with the associated eigenvectors $\bigg(-\frac{\alpha+\beta+\eta+ \sqrt{\alpha^2+ \beta^2+  4 \eta^2 + 2\alpha \beta+ 2\alpha \eta+  2\beta \eta}}{\eta},$ $ 1,$ $ 1,$ $1 \bigg)$, $ \left(-\frac{\alpha+\beta+\eta- \sqrt{\alpha^2+ \beta^2+  4 \eta^2 + 2\alpha \beta+ 2\alpha \eta+  2\beta \eta}}{\eta}, 1, 1,1 \right)$, $(0,1,-1,0)$, $(0,1,0,-1)$ respectively.}
	\end{example}


	\section{ Universal Adjacency Spectrum of Some Proper Power Graphs } 
	Here we  discuss universal adjacency  spectrum of proper power graphs of $\mathbb{Z}_n$, ${D}_{n}$ and ${Q}_n$. By  proper divisors of $n$ we mean all the positive divisors of $n$ which are strictly less than $n$. Here we take $d_1$, $d_2$, ...,$d_m$ as the  proper divisors of $n$. So $m=t-1$.
	
	\begin{theorem}\label{theorem7.1}
		The  eigenvalues of $U(\mathscr{P^*}(\mathbb{Z}_n ))$ and corresponding eigenvectors are as given below:
		\begin{enumerate}
			\item The eigenvalue 	$\Lambda_{i}$ = $-\alpha +\left( \phi(\frac{n}{d_i})-1  +\sum\limits_{d_{i} \sim d_{j}}\phi(\frac{n}{d_{j}})  \right)\beta +\gamma$  with multiplicity $\left( \phi(\frac{n}{d_{i}}) - 1 \right)$. For fixed $i$ $(1 \leq i \leq m)$, the linearly independent eigenvector associated with $\Lambda_i$ are $X_{i,r}$ = $(x_1,x_2...,x_m)^T$, where 
			\begin{equation*}
				\textbf{x}_k =\begin{cases}
					\textbf{e}_{\phi(\frac{n}{d_i}),r} ^{T} ,              & \text{if~~}  k=i,\\
					0             ,    & \text{otherwise},
				\end{cases} \text{~~where i= 1,2...
					m~~} \text{ and r= 2,3,...,$\phi \left(\frac{n}{d_i}\right)$~~}
			\end{equation*}	\\
			
			\item And the rest of  $m$ eigenvalues of $U(\mathscr{P^*}(\mathbb{Z}_n ))$ are the eigenvalues of  $\mathbb{K}$  with  corresponding eigenvectors   $(\nu_{i,1} \sqrt\frac{\phi(\frac{n}{d_m})}{\phi(\frac{n}{d_1})} \textbf{j} _{\phi(\frac{n}{d_1})} ,$ $\nu_{i,2} \sqrt\frac{\phi(\frac{n}{d_m})}{\phi(\frac{n}{d_2})} \textbf{j} _{\phi(\frac{n}{d_2})},$ $\ldots,$ $ \nu_{i,m-1} \sqrt\frac{\phi(\frac{n}{d_m})}{\phi(\frac{n}{d_{m-1}})} \textbf{j} _{\phi(\frac{n}{d_{m-1}})}  ,$ $ \nu_{i,m} \textbf{j}_{\phi(\frac{n}{d_m})}  ) ^{T} $ 
			where $\nu_{i}$ = $( \nu_{i,1} , \nu_{i,2},... ,\nu_{i,m} )^{T} $ and $(\lambda_{i} ,\nu_{i})$ is an eigenpair of $\mathbb{K}$ which is given below:	  	
			\begin{equation}\label{symmetricmatrix}
				\mathbb{K} =
				\begin{pmatrix}
					\kappa_{11}&\theta_{1,2} \sqrt{\phi(\frac{n}{d_1}) \phi(\frac{n}{d_2})}&\cdots &\theta_{1,k} \sqrt{\phi(\frac{n}{d_1}) \phi(\frac{n}{d_m})} \\
					\theta_{2,1} \sqrt{\phi(\frac{n}{d_2}) \phi(\frac{n}{d_1})}&\kappa_{22}&\cdots &\theta_{2,m} \sqrt{\phi(\frac{n}{d_2}) \phi(\frac{n}{d_m})} \\
					\vdots & \vdots & \ddots & \vdots\\
					\theta_{k,1} \sqrt{\phi(\frac{n}{d_m}) \phi(\frac{n}{d_1})}&\theta_{k,2} \sqrt{\phi(\frac{n}{d_m}) \phi(\frac{n}{d_2})}&\cdots &\kappa_{mm}
				\end{pmatrix}
			\end{equation}
			where $\kappa_{ii}$ = $\alpha \left( \phi(\frac{n}{d_{i}} ) -1\right) +\beta \left( \phi(\frac{n}{d_i})-1 + \sum\limits_{d_{i} \sim d_{j}} \phi(\frac{n}{d_{j}} ) \right)+\gamma + \phi\left(\frac{n}{d_{i}} \right)\eta $  and 
			\begin{equation*}
				\theta_{i,j} =\begin{cases}
					\alpha+\eta,              & \text{if~~}  d_i |d_j \text{~or~} d_j |d_i , \\ 
					\eta,      & \text{else}, 
				\end{cases}   1\leq i,j \leq m.
			\end{equation*}
		\end{enumerate}
	\end{theorem}

	\begin{proof} 
		Let $H_{d_1}, H_{d_2}, ..., H_{d_m}$  and $\Omega_n$ be as given in Section \ref{2}. Then $H_{d_1} \cup H_{d_2} \cup ...\cup H_{d_m}$ forms a partition of  $V(\mathscr{P^*}(\mathbb{Z}_n ))$, and  $\mathscr{P^*}(\mathbb{Z}_n )$ can be expressed as $\Omega_n-\{0\}$-join of  $H_{d_1}, H_{d_2}, ..., H_{d_m}$. Now the required result can be obtained by Theorem \ref{theorem3.2}.
	\end{proof}
	The next result can be shown in the similar way as in Theorem  \ref{theorem4.4}. 
	
	\begin{theorem}\label{theorem7.2}
		The  eigenvalues of  $U(\overline{\mathscr{P^*}({Z}_{n} )})$  and corresponding eigenvectors are as given below:
		\begin{enumerate}
			\item The eigenvalue $\Lambda_{i}$ =	$\beta \left( n -\phi(\frac{n}{d_i}) - \sum\limits_{d_{i} \sim d_{j}} \phi(\frac{n}{d_{j}}) \right) +\gamma $     with multiplicity $\left(\phi(\frac{n}{d_{i}}) -1 \right)$  with the associated eigenvector $\textbf{X}_{i,r}$ = $(\textbf{x}_1, \textbf{x}_2 ,\ldots,\textbf{x}_m) ^{T}$, where 
			\begin{equation*}
				\textbf{x}_k =\begin{cases}
					\textbf{e}_{\phi(\frac{n}{d_i}),r} ^{T},             & \text{if~~}  k=i,\\
					0            ,     & \text{else},
				\end{cases} \text{~~where i= 1,2,\ldots,
					m~~} \text{ and r= 2,3,\ldots, $\phi(\frac{n}{d_i})$ ~~}
			\end{equation*}
			\item And the rest of  $m$ eigenvalues of  $U(\overline{\mathscr{P^*}({Z}_{n} )})$ are the eigenvalues of   $\mathbb{K}$  with  corresponding eigenvectors   $(\nu_{i,1} \sqrt\frac{\phi(\frac{n}{d_m})}{\phi(\frac{n}{d_1})} \textbf{j} _{\phi(\frac{n}{d_1})} , \nu_{i,2} \sqrt\frac{\phi(\frac{n}{d_m})}{\phi(\frac{n}{d_2})} \textbf{j} _{\phi(\frac{n}{d_2})},..., $  $ \nu_{i,m-1} \sqrt\frac{\phi(\frac{n}{d_m})}{\phi(\frac{n}{d_{m-1}})} \textbf{j} _{\phi(\frac{n}{d_{m-1}})}  , \nu_{i,m} \textbf{j}_{\phi(\frac{n}{d_m})}  ) ^{T} $ 
			where $\nu_{i}$ = $( \nu_{i,1} , \nu_{i,2},\ldots,\nu_{i,m} )^{T} $ and $(\lambda_{i} ,\nu_{i})$ is an eigenpair of $\mathbb{K}$, 
			where  
			\begin{equation*}
				\kappa_{ij} =\begin{cases}
					(\eta + \alpha)\sqrt{\phi(\frac{n}{d_i}){\phi(\frac{n}{d_j})}}             & \text{if neither~~}  d_i |d_j \text{~nor~}  d_j |d_i \\ 
					\eta\sqrt{\phi(\frac{n}{d_i}){\phi(\frac{n}{d_j})}}                  & \text{otherwise~~~~~~~}  
				\end{cases} \text{i $\neq$ j}
			\end{equation*}
			\\and \\
			\begin{equation*}
				\kappa_{ii} =\begin{cases}
					\beta \left( n - \phi(\frac{n}{d_i})  +\sum\limits_{d_{i} \sim d_{j}} \phi(\frac{n}{d_{j}}) \right) +\gamma +\eta \phi\left(\frac{n}{d_{i}}\right)                      
				\end{cases}
			\end{equation*}
		\end{enumerate}
	\end{theorem}

	\begin{theorem}\label{theoremDn}
		The  eigenvalues of  $U(\mathscr{P^*}(D_n ))$  and corresponding eigenvectors are as given below:
		\begin{enumerate}
			\item The eigenvalue $\Lambda_{i} = -\alpha + \left( \phi(\frac{n}{d_i}) -1+ \sum\limits_{d_{i} \sim d_{j}} \phi\left(\frac{n}{d_{j}}\right) \right)\beta +\gamma$ with multiplicity $\left( \phi(\frac{n}{d_{i}}) - 1\right)$ with the associated eigenvector $\textbf{X}_{i,r}$ = $(\textbf{x}_1 , \textbf{x}_2 ,\ldots,\textbf{x}_m,\textbf{x}_R) ^{T}$, where 
			\begin{equation*}
				\textbf{x}_k =\begin{cases}
					\textbf{e}_{\phi (\frac{n}{d_i}),r} ^{T}  ,             & \text{if~~}  k=i,\\
					0           ,      & \text{otherwise},
				\end{cases} \text{~~where i= 1,2,\ldots,	m~~} \text{ and r= 2,3,\ldots, $\phi \left(\frac{n}{d_i} \right)$ }
			\end{equation*}
			
			\item The eigenvalue $\Lambda_{R}$ = $ \gamma$  with multiplicity $(n-1)$ with the associated eigenvector $\textbf{X}_{i,r}$ = $(\textbf{x}_1, \textbf{x}_2,\ldots,\textbf{x}_m,\textbf{x}_R) ^{T}$, where 
			\begin{equation*}
				\textbf{x}_k =\begin{cases}
					\textbf{e}_{n,r} ^{T},             & \text{if~~}  k=R,\\
					0,     & \text{otherwise},
				\end{cases}
				\text{ and r = 2,3, \ldots,n }
			\end{equation*}

			\item  And the   rest of $m+1$ eigenvalues of $U(\mathscr{P^*}(D_n ))$ are the eigenvalues of  $\mathbb{K}$  with associated eigenvectors   $(\nu_{i,1} \sqrt\frac{n}{\phi(\frac{n}{d_1})} \textbf{j} _{\phi(\frac{n}{d_{1}})} , \nu_{i,2} \sqrt\frac{n}{\phi(\frac{n}{d_2})} \textbf{j} _{\phi(\frac{n}{d_{2}})},\ldots,$ $\nu_{i,m} \sqrt\frac{n}{\phi(\frac{n}{d_{m}})} \textbf{j} _{\phi(\frac{n}{d_{m}})}  , \nu_{i,R} \textbf{j}_ n ) ^{T} $ 
			where $\nu_{i}$ = $( \nu_{i,1} , \nu_{i,2},\ldots, \nu_{i,m},\nu_{i,R} )^{T} $ and $(\lambda_{i} , \nu_{i})$ is an eigenpair of $\mathbb{K}$ which is given below:

			\begin{equation}\label{symmetricmatrix15}
				\mathbb{K} =
				\begin{pmatrix}
					\kappa_{11}&\theta_{1,2} \sqrt{\phi(\frac{n}{d_1}){\phi(\frac{n}{d_2})}}&\cdots &\theta_{1,m}  \sqrt{\phi(\frac{n}{d_1}){\phi(\frac{n}{d_m})}} &\theta_{1,R}\sqrt{\phi(\frac{n}{d_1}).n} \\
					\theta_{2,1} \sqrt{\phi(\frac{n}{d_2}){\phi(\frac{n}{d_1})}}&\kappa_{22}&\cdots &\theta_{2,m} \sqrt{\phi(\frac{n}{d_2}){\phi(\frac{n}{d_m})}} &\theta_{2,R}\sqrt{\phi(\frac{n}{d_2}).n}  \\
					\vdots & \vdots & \ddots & \vdots\\
					\theta_{m,1} \sqrt{\phi(\frac{n}{d_m}){\phi(\frac{n}{d_1})}}&\theta_{m,2} \sqrt{\phi(\frac{n}{d_m}){\phi(\frac{n}{d_2})}}&\cdots &\kappa_{mm}& \theta_{m,R}\sqrt{\phi(\frac{n}{d_m}).n} \\
					\theta_{R,1}\sqrt{\phi(\frac{n}{d_1}).n}& \theta_{R,2}\sqrt{\phi(\frac{n}{d_2}).n}& \cdots&\theta_{R,m}\sqrt{\phi(\frac{n}{d_m}).n}& \kappa_{RR}
				\end{pmatrix}
			\end{equation}
			
			where
			\begin{equation*}
				\kappa_{ii} =\begin{cases} \alpha \left( \phi(\frac{n}{d_{i}}) -1  \right) +\beta \left(\phi(\frac{n}{d_i}) -1+ 
					\sum\limits_{d_{i} \sim d_{j}} \phi(\frac{n}{d_{j}} ) \right)+\gamma + \phi\left(\frac{n}{d_{i}} \right)\eta,  & \text{when $i$ $\neq$ $R$ }\\
					
					\gamma +\eta n, &\text {i= R}
				\end{cases}
			\end{equation*}
			
			\begin{equation*}
				\theta_{i,j} =\begin{cases}
					\alpha +	\eta,             & \text{if~~}  d_i |d_j \text{~or~} d_j |d_i,  \\ 
					\eta,      & \text{else}  ,
				\end{cases}   1\leq i,j \leq m. 
			\end{equation*} 
			
			\begin{equation*}
				\theta_{R,i} = \theta_{i,R}= \begin{cases}
					
					\eta
				\end{cases}  \text{for } 1\leq i \leq m.
			\end{equation*} 
		\end{enumerate}
	\end{theorem}
	
	\begin{proof} 	Let $S_{d_1}, S_{d_2}, ..., S_{d_m}$  and $\Omega_n '$ be as given in Section \ref{section5}. Then $S_{d_1} \cup S_{d_2} \cup ...\cup S_{d_m}$ forms a partition of  $V(\mathscr{P^*}(D_n ))$, and  $\mathscr{P^*}(D_n )$ can be expressed as $\Omega_n-\{0\}$-join of  $S_{d_1}, S_{d_2}, ..., S_{d_m}$. Now the required result can be obtained by Theorem  \ref{theorem5.1}.
	\end{proof}
	
	When $n = p^r$, the next theorem provides the full universal adjacency spectrum of $\mathscr{P^*}({D}_n )$.

	\begin{theorem}\label{Theorem7.4}
		Let $n=p^r$,   the   eigenvalues of  $U(\mathscr{P^*}({D}_n ))$    and corresponding eigenvectors are as given below:
		\begin{enumerate}
			\item  The eigenvalue $\Lambda_{i1}$ = $ - \alpha +( p^r -2)\beta + \gamma $  with multiplicity  $(p^r-2)$  and with the associated eigenvector
			$X_{i}$ = $(-1,x_{i2}, x_{i3},...,x_{i(p^r-1)},\underbrace{0,0,...,0}_{p^r -times})$, 	where
			
			\begin{equation*}
				x_{ij}= \begin{cases}
					1 & \text{if~~} i+1=j\\
					0 & \text {otherwise.}
				\end{cases} 	\text{~~~i= 1,2,...$(p^r-2)$}
			\end{equation*} 
			
			\item The eigenvalue $\Lambda_{i2}$ = $\gamma $ with multiplicity $(p^r-1)$ and with the associated eigenvector  $Y_{i}$ = $(\underbrace{0,0,...,0}_{p^r-1 times} , -1,y_{i2},y_{i3},...,y_{i(p^r)})$, where 
			
			\begin{equation*}
				y_{ij}= \begin{cases}
					1 & \text{if~~} i+1=j\\
					0 & \text {otherwise.}
				\end{cases} 	\text{~~~i= 1,2,...,$(p^r-1)$}
			\end{equation*}
			
			\item And the rest of  two eigenvalues are  $\lambda_1$ and $\lambda_2$ of $\mathbb{K}$ with the corresponding eigenvectors 	$(\nu_{i,1} \sqrt\frac{p^r}{p^r-1} \textbf{j}_{p^r-1}, \nu_{i,2}  \textbf{j}_{p^r}) ^{T} $, where $\nu_{i}$ =
			$( \nu_{i,1}, \nu_{i,2} )^{T} $ and $(\lambda_{i}, \nu_{i})$ is an eigenpair of $\mathbb{K}$ which is given below: 
			\begin{equation*}\label{symmetricmatrix}
				\mathbb{K} =
				\begin{pmatrix}
					(p^r-2)\alpha+(p^r-2)\beta+\gamma+(p^r-1)\eta&\eta\sqrt{p^r(p^r-1)}  \\
					\eta\sqrt{(p^r-1)p^r}&\gamma +(p^r)\eta \\
				\end{pmatrix} 
			\end{equation*}
			where 
			$\lambda_1$= $\frac{(p^r-2)\alpha+(p^r-2)\beta+2 \gamma+(2p^r-1)\eta + 2\sqrt{((p^r-2)\alpha+(p^r-2)\beta+\gamma+(p^r-1)(\gamma+p^r \eta))}}{2}$ \\
			and 
			$\lambda_2$= $\frac{(p^r-2)\alpha+(p^r-2)\beta+2 \gamma+(2p^r-1)\eta - 2\sqrt{((p^r-2)\alpha+(p^r-2)\beta+\gamma+(p^r-1)(\gamma+p^r \eta))}}{2}$ 
		\end{enumerate}
	\end{theorem}
	\begin{proof}
		Let  $H_1$ is the set of all rotations except identity and $H_2$ is the set of all reflections. Then $H_1 \cup H_2$ forms a partition of $V(\mathscr{P^*}({D}_n ))$. Now $H_1$ and $H_2$ induce $K_{(p^r-1)}$ and $\overline{K_{p^r}}$ respectively in $\mathscr{P^*}({D}_n )$. The graph $\mathscr{P^*}({D}_n )$ is the $\overline{K_2}$-join of $H_1$ and $H_2$. Now the required  result can be obtained by Theorem \ref{theorem3.1}.
	\end{proof}

	\begin{theorem}
		The  eigenvalues of   $\overline{\mathscr{P^*}({D}_n )}$    and corresponding eigenvectors are as given below:
		\begin{enumerate}
			\item  The eigenvalue $\Lambda_{i}$ =	$\beta \left( 2n-1 - \phi(\frac{n}{d_{i}}) -  \sum\limits_{d_{i} \sim d_{j}} \phi(\frac{n}{d_{j}}) \right) +\gamma $     with multiplicity $\left(\phi(\frac{n}{d_{i}}) -1\right)$  with the associated eigenvector $\textbf{X}_{i,r}$ = $(\textbf{x}_1 , \textbf{x}_2 ,...,\textbf{x}_m,\textbf{x}_R) ^{T}$, where 
			\begin{equation*}
				\textbf{x}_k =\begin{cases}
					\textbf{e}_{\phi (\frac{n}{d_i}),r} ^{T}  ,             & \text{if~~}  k=i,\\
					0           ,      & \text{otherwise},
				\end{cases} \text{~~where i= 1,2,\ldots,m~~} \text{ and r= 2,3,\ldots, $\phi \left(\frac{n}{d_i}\right)$ }
			\end{equation*}
			
			\item The eigenvalue $\Lambda_R$ = $-\alpha+(2n-2)\beta + \gamma$ with multiplicity $(n-1)$ with the associated eigenvector $\textbf{X}_{i,j}$ = $(\textbf{x}_1, \textbf{x}_2,...\textbf{x}_m,\textbf{x}_R) ^{T}$, where 
			\begin{equation*}
				\textbf{x}_k =\begin{cases}
					\textbf{e}_{n,r} ^{T},             & \text{if~~}  k=R,\\
					0,     & \text{otherwise},
				\end{cases}
				\text{ and r = 2,3,\ldots,n }
			\end{equation*}
			
			\item And the rest of  $m+1$ eigenvalues are the eigenvalues of   $\mathbb{K}$   with associated eigenvectors   $(\nu_{i,1} \sqrt\frac{n}{\phi(\frac{n}{d_1})} \textbf{j} _{\phi(\frac{n}{d_{1}})}, \nu_{i,2} \sqrt\frac{n}{\phi(\frac{n}{d_2})} \textbf{j} _{\phi(\frac{n}{d_{2}})}, \ldots, \nu_{i,m} \sqrt\frac{n}{\phi(\frac{n}{d_{m}})} \textbf{j} _{\phi(\frac{n}{d_{m}})},$ $ \nu_{i,R} \textbf{j}_ n ) ^{T} $ 
			where $\nu_{i}$ = $( \nu_{i,1}, \nu_{i,2},\ldots, \nu_{i,m}, \nu_{i,R} )^{T} $ and $(\lambda_{i} , \nu_{i})$ is an eigenpair of $\mathbb{K}$ which is given below:

			\begin{equation}\label{symmetricmatrix}
				\mathbb{K} = \scriptsize
				\begin{pmatrix}
					\kappa_{1}&\theta_{1,2} \sqrt{\phi(\frac{n}{d_1}){\phi(\frac{n}{d_2})}}&\cdots &\theta_{1,m}  \sqrt{\phi(\frac{n}{d_1}){\phi(\frac{n}{d_m})}} &\theta_{1,R}\sqrt{\phi(\frac{n}{d_1}).n} \\
					\theta_{2,1} \sqrt{\phi(\frac{n}{d_2}){\phi(\frac{n}{d_1})}}&\kappa_{2}&\cdots &\theta_{2,m} \sqrt{\phi(\frac{n}{d_2}){\phi(\frac{n}{d_m})}} &\theta_{2,R}\sqrt{\phi(\frac{n}{d_2}).n}  \\
					\vdots & \vdots & \ddots & \vdots\\
					\theta_{m,1} \sqrt{\phi(\frac{n}{d_m}){\phi(\frac{n}{d_1})}}&\theta_{m,2} \sqrt{\phi(\frac{n}{d_m}){\phi(\frac{n}{d_2})}}&\cdots &\kappa_{k}& \theta_{k,R}\sqrt{\phi(\frac{n}{d_m}).n} \\
					\theta_{R,1}\sqrt{\phi(\frac{n}{d_1}).n}& \theta_{R,2}\sqrt{\phi(\frac{n}{d_2}).n}& \cdots&\theta_{R,m}\sqrt{\phi(\frac{n}{d_m}).n}& \kappa_{R}
				\end{pmatrix}
			\end{equation}
			where	
			\begin{equation*}
				\kappa_{i} =\begin{cases}
					\beta \left( 2n-1 -\phi(\frac{n}{d_{i}})- \sum\limits_{d_{i} \sim d_{j}} \phi(\frac{n}{d_{j}}) \right) +\gamma +\eta \phi\left(\frac{n}{d_{i}}\right)               ,\\ 
					
				\end{cases}  1 \leq i \leq m.
			\end{equation*}

			\begin{equation*}
				\kappa_{R} =  (n-1)\alpha + (2n-2)\beta +\gamma +\eta n; 
			\end{equation*}
			
			\begin{equation*}
				\theta_{i,j} =\begin{cases}
					\eta + \alpha  ,             & \text{if neither~~}  d_i |d_j \text{~nor~}  d_j |d_i  , \\ 
					\eta           ,      & \text{else}, 
				\end{cases}  1\leq i,j \leq m;
			\end{equation*} 
			
			\begin{equation*}
				and~~~	\theta_{R,i} = \theta_{i,R}= \begin{cases}
					
					\eta +\alpha  
				\end{cases}  1\leq i \leq m;
			\end{equation*} 
		\end{enumerate}
		
	\end{theorem}

	\begin{proof}
		Analogous to that of  Theorem \ref{theorem5.2}.

	\end{proof}

	\begin{theorem}
		Let $n=p^r$,	the  eigenvalues of   $\overline{\mathscr{P^*}({D}_n )}$    and corresponding eigenvectors are as given below:
		\begin{enumerate}
			\item The eigenvalue $\Lambda_{i1}$ = $  p^r \beta + \gamma $  with multiplicity  $(p^r-2)$  and with the associated eigenvector
			$X_{i}$ = $(-1,x_{i2}, x_{i3},...,x_{i(p^r-1)},\underbrace{0,0,...,0}_{p^r -times})$, 	where	
			\begin{equation*}
				x_{ij}= \begin{cases}
					1 & \text{if~~} i+1=j\\
					0 & \text {otherwise.}
				\end{cases} 	\text{~~~i= 1,2,...$(p^r-2)$}
			\end{equation*} 
			
			\item The eigenvalue $\Lambda_{i2}$ = $-\alpha+(2p^r -2)\beta+\gamma $ with multiplicity $(p^r-1)$ and with the associated eigenvector $Y_{i}$ = $(\underbrace{0,0,...,0}_{p^r-1 times} , -1,y_{i2},y_{i3},...,y_{i(p^r)})$, where 
			
			\begin{equation*}
				y_{ij}= \begin{cases}
					1 & \text{if~~} i+1=j\\
					0 & \text {otherwise.}
				\end{cases} 	\text{~~~i= 1,2,...,$(p^r-1)$}
			\end{equation*}
			
			\item And the rest of  two eigenvalues are $\lambda_1$ and $\lambda_2$ of  $\mathbb{K}$ with corresponding eigenvectors 	$(\nu_{i,1} \sqrt\frac{p^r}{p^r-1} \textbf{j}_{p^r-1}, \nu_{i,2}  \textbf{j}_{p^r}) ^{T} $ 	where $\nu_{i}$ =
			$( \nu_{i,1}, \nu_{i,2} )^{T} $ and $(\lambda_{i},\nu_{i})$ is an eigenpair of $\mathbb{K}$, where
			\begin{equation*}\label{symmetricmatrix}
				\mathbb{K} =
				\begin{pmatrix}
					p^r\beta+\gamma+(p^r-1)\eta&(\alpha+\eta)\sqrt{p^r(p^r-1)}  \\
					(\alpha+\eta)\sqrt{(p^r-1)p^r}&(p^r-1)\alpha+(2p^r-2)\beta+\gamma +(p^r)\eta \\
				\end{pmatrix} 
			\end{equation*}
			
		\end{enumerate}
	\end{theorem}
	\begin{proof}
		Here, we use the same partition for vertex set of $\overline{\mathscr{P^*} ({D}_n )}$  that we did for $\mathscr{P^*} ({D}_n )$ in Theorem \ref{Theorem7.4}. Since $U(\overline{\mathscr{P^*}(D_n)})= \alpha A (\overline{\mathscr{P^*}(D_n)})+ \beta D(\overline{\mathscr{P^*}(D_n)})+ \gamma I_n + \eta J_n= - \alpha A(\mathscr{P^*}(D_n)) - \beta D(\mathscr{P^*}(D_n)) + (\gamma +\beta(n-1)-\alpha)I_n +(\eta+ \alpha)J_n$, the result is obtained by Theorem \ref{Theorem7.4} substituting  $\alpha= -\alpha, \beta=-\beta, \gamma= \gamma +\beta(n-1)-\alpha$ and $\eta=\eta+ \alpha$.
	\end{proof}

	\begin{theorem}\label{theorem7.7}
		For $n = 2^r$, the  eigenvalues of   $U(\mathscr{P^*}({Q}_n ))$    and corresponding eigenvectors are as given below:
		\begin{enumerate}
			
			\item The eigenvalue $-\alpha+(2n-2)\beta+\gamma$ with multiplicity $(2n-3)$ with the associated eigenvector $X_i = (0,1,x_{i2},x_{i3},...,x_{i(2n-2)},\underbrace{0,0,...,0}_{2n ~times})$,
			where    \begin{equation*}
				x_{ij}= \begin{cases}
					-1 & \text{if~~} i=j\\
					0 & \text {otherwise.}
				\end{cases} 	\text{~~~i= 2,...,$(2n-2)$}
			\end{equation*}
			
			\item  The eigenvalue $-\alpha + 2\beta + \gamma$ with  multiplicity $n$ with the associated eigenvector \break $\textbf{Y}_{i}$ = $(\textbf{y}_1 , \textbf{y}_2 ,...,\textbf{y}_{(n+1)},\textbf{y}_{(n+2)}) ^{T}$, where 
			$\textbf{y}_1, \textbf{y}_2$ are $0$ vectors of dimension $1$, $(2n-2)$ respectively and
			\begin{equation*}
				\textbf{y}_l =\begin{cases}
					\textbf{e}_{2,2} ^{T}  ,             & \text{if~~}  l=i,\\
					0           ,      & \text{otherwise},
				\end{cases} \text{~~where i= 3,2...,	(n+2)~~}
			\end{equation*}
			
			\item The rest of  $(n+2)$ eigenvalues of $U(\mathscr{P^*}({Q}_n ))$ are the eigenvalues of $\mathbb{K}$ with corresponding eigenvectors   $(\nu_{i,1} \sqrt{2} \textbf{j} _{1} ,$ $ \nu_{i,2} \sqrt\frac{2}{2n-2} \textbf{j} _{2n-2},\nu_{i,3} \textbf{j} _{2},\ldots,$ $v_{i, n+1} \textbf{j} _{2}  ,$ $ \nu_{i,n+2} \textbf{j}_{2}  ) ^{T} $ where $\nu_{i}$ = $( \nu_{i,1} , \nu_{i,2},... ,\nu_{i,n+2} )^{T} $  and $(\lambda_i, \nu_i)$ is an eigenpair of $\mathbb{K}$ which is given below:
			\begin{equation*}\label{symmetricmatrix}
				\mathbb{K} = \scriptsize
				\begin{pmatrix}
					
					\kappa_1&(\alpha+\eta)\sqrt{(2n-2)} &(\alpha+\eta)\sqrt{2}&(\alpha+\eta)\sqrt{2}&\cdots&(\alpha+\eta)\sqrt{2}\\
					(\alpha+\eta)\sqrt{(2n-2)}&\kappa_2&\eta \sqrt{2(2n-2)} &\eta\sqrt{2(2n-2)}&\cdots&\eta\sqrt{2(2n-2)} \\
					
					(\alpha+\eta)\sqrt{2}&\eta\sqrt{2(2n-2)}&\alpha+2\beta+\gamma+2\eta &\eta\sqrt{2.2}&\cdots&\eta\sqrt{2.2} \\
					\vdots&\vdots&\vdots&\ddots&\vdots&\vdots\\	
					(\alpha+\eta)\sqrt{2} &\eta\sqrt{2(2n-2)} &\eta\sqrt{2.2}&\eta\sqrt{2.2}&\cdots&\alpha+2\beta+\gamma+2\eta \\
					
				\end{pmatrix}
			\end{equation*}
			
		\end{enumerate}
		where $\kappa_1$ = $(4n-2)\beta+ \gamma+ \eta$ and $\kappa_2 = (2n-3)\alpha +(2n-2)\beta+ \gamma + (2n-2)\eta$. Moreover $\mathbb{K}$ has three simple 
		eigenvalues and the eigenvalue $\alpha+2\beta+\gamma$ with multiplicity $(n-1)$. 
	\end{theorem}
	
	\begin{proof} 
		Here, we use the same partition for vertex set of  $\mathscr{P^*} ({Q}_n )$ that we did for  $\mathscr{P} ({Q}_n )$ in Section \ref{section6}, except  $T_1$=$\{ a^n\}$. Then proof is  analogous to that of Theorem \ref{theorem6.1}.   
	\end{proof}

	\begin{theorem}\label{theorem7.8}
		For $n = 2^r$,  the  eigenvalues of   $\overline{\mathscr{P^*}({Q}_n )}$     and corresponding eigenvectors are as given below:
		\begin{enumerate}
			
			\item The eigenvalue $(2n)\beta+\gamma$ with multiplicity $(2n-3)$ with the associated eigenvector $X_i = (0,1,x_{i2},x_{i3},...,x_{i(2n-2)},\underbrace{0,0,...,0}_{2n ~times})$,
			where    \begin{equation*}
				x_{ij}= \begin{cases}
					-1 & \text{if~~} i=j\\
					0 & \text {otherwise.}
				\end{cases} 	\text{~~~i= 2,...,$(2n-2)$}
			\end{equation*}
			
			\item  The eigenvalue $(4n-4)\beta + \gamma$ with multiplicity $n$ with the associated eigenvector\break $\textbf{Y}_{i}$ = $(\textbf{y}_1 , \textbf{y}_2 ,...,\textbf{y}_{(n+1)},\textbf{y}_{(n+2)}) ^{T}$, where  $\textbf{y}_1, \textbf{y}_2$ are $0$ vectors of dimension $1$, $(2n-2)$ respectively and
			\begin{equation*}
				\textbf{y}_l =\begin{cases}
					\textbf{e}_{2,2} ^{T}  ,             & \text{if~~}  l=i,\\
					0           ,      & \text{otherwise},
				\end{cases} \text{~~where i= 3,2...,	(n+2)~~}
			\end{equation*}
			
			\item The rest of  $(n+2)$ eigenvalues of $U(\overline{\mathscr{P^*}(Q_{n} )})$ are the eigenvalues of $\mathbb{K}$ with  the corresponding eigenvectors   $(\nu_{i,1} \sqrt{2} \textbf{j} _{1} ,$ $\nu_{i,2} \sqrt\frac{2}{2n-2} \textbf{j} _{2n-2},\nu_{i,3} \textbf{j} _{2},\ldots, $ $\nu_{i, n+1} \textbf{j} _{2}  , \nu_{i,n+2} \textbf{j}_{2}  ) ^{T} $ where $\nu_{i}$ = $( \nu_{i,1} , \nu_{i,2},\ldots ,\nu_{i,n+2} )^{T} $  and $(\lambda_i, \nu_i)$ is an eigenpair of $\mathbb{K}$ which is given below:
			
			\begin{equation}\label{symmetricmatrix 17}
				\mathbb{K} =
				\scriptsize
				\begin{pmatrix}
					
					\kappa_1&\eta\sqrt{2(2n-2)} &\eta\sqrt{2.2}&\eta\sqrt{2.2}&\cdots&\eta\sqrt{2.2}\\
					\eta\sqrt{2.(2n-2)}&\kappa_2&(\alpha+\eta) \sqrt{2(2n-2)} &(\alpha+\eta)\sqrt{2(2n-2)}&\cdots&(\alpha+\eta)\sqrt{2(2n-2)} \\
					
					\eta\sqrt{2.2}&(\alpha+\eta)\sqrt{2(2n-2)}&(4n-4)\beta+\gamma+2\eta &(\alpha+\eta)\sqrt{2.2}&\cdots&(\alpha+\eta)\sqrt{2.2} \\
					\vdots&\vdots&\vdots&\ddots&\vdots&\vdots\\	
					\eta\sqrt{2.2} &(\alpha+\eta)\sqrt{2(2n-2)} &(\alpha+\eta)\sqrt{2.2}&(\alpha+\eta)\sqrt{2.2}&\cdots&(4n-4)\beta+\gamma+2\eta \\
					
				\end{pmatrix}
			\end{equation}
			
			where $\kappa_1$ = $ \gamma+ \eta$ and $\kappa_2 = (2n)\beta+ \gamma + (2n-2)\eta$. Moreover,  $\mathbb{K}$ has three simple  eigenvalues and the eigenvalue $-2\alpha+(4n-3)\beta+\gamma$ with multiplicity  $(n-1)$. 
		\end{enumerate} 
		
	\end{theorem}
	
	\begin{proof}
		Here we consider the same partition for  $V(\overline{\mathscr{P^*} ({Q}_n )})$  that we have taken for $\mathscr{P^*} ({Q}_n )$ in Theorem \ref{theorem7.7}.  
		Since $U(\overline{\mathscr{P^*}(Q_n)})= \alpha A (\overline{\mathscr{P^*}(Q_n)})+ \beta D(\overline{\mathscr{P^*}(Q_n)})+ \gamma I_n + \eta J_n= - \alpha A(\mathscr{P^*}(Q_n)) - \beta D(\mathscr{P^*}(Q_n)) + (\gamma +\beta(n-1)-\alpha)I_n +(\eta+ \alpha)J_n$, the result is  obtained  from Theorem \ref{theorem7.7} by substituting  $\alpha= -\alpha, \beta=-\beta, \gamma= \gamma +\beta(n-1)-\alpha$ and $\eta=\eta+ \alpha$.
	\end{proof}
	
	\noindent \textbf{Data Availability:} Data sharing is not applicable to this article as no datasets were generated or analysed during the current study.
	\\
	
	\noindent \textbf{Conflict of Interest:} The authors declare they have no competing interests.

	

\end{document}